\documentclass[letter,11pt,leqno,english]{amsart}
\RequirePackage{doi}
\usepackage[T1]{fontenc}
\usepackage{tgbonum}

\usepackage{pdfsync,hyperref,calc,enumerate,amssymb,amscd,color}
\usepackage[all]{xy}
\usepackage{color,soul}

\newcommand{\version}[1] 
{
\begin{center}
Last edited on #1\\
Last compiled on \today\\
file name: \jobname
\end{center}
}

\newcounter{commentcounter}

\theoremstyle{plain}
\newtheorem{theorem}{Theorem}[section]

\newtheorem{lemma}[theorem]{Lemma}

\newtheorem{corollary}[theorem]{Corollary}
\newtheorem{proposition}[theorem]{Proposition}

\newtheorem*{theorem*}{Theorem}
\newtheorem*{mtheorem*}{Main Theorem}

\theoremstyle{definition}
\newtheorem{definition}[theorem]{Definition}

\theoremstyle{remark}
\newtheorem*{summary*}{Summary}

\makeatletter\let\c@equation=\c@theorem\makeatother

\DeclareMathAlphabet{\matheurm}{U}{eur}{m}{n}

\DeclareMathOperator{\Coker}{Coker}

\DeclareMathOperator{\Gr}{Gr}

\DeclareMathOperator{\Image}{Im}

\DeclareMathOperator{\Ker}{Ker}

\DeclareMathOperator{\Map}{Map}

\DeclareMathOperator{\Tor}{Tor}

\DeclareMathOperator{\Hom}{\textup{Hom}}

\DeclareMathOperator{\Irr}{\textup{Irr}}

\newcommand{\IC}{\mathbb{C}}
\newcommand{\IQ}{\mathbb{Q}}

\newcommand{\IZ}{\mathbb{Z}}
\newcommand{\IP}{\mathbb{P}}
\newcommand{\cala}{\mathcal{A}}
\newcommand{\calf}{\mathcal{F}}
\newcommand{\calp}{\mathcal{P}}

\makeatletter
\newcommand{\colim@}[2]{%
  \vtop{\m@th\iequation{##\cr
    \hfil$#1\operator@font colim$\hfil\cr
    \noequation{\nointerlineskip\kern1.5\ex@}#2\cr
    \noequation{\nointerlineskip\kern-\ex@}\cr}}%
}
\newcommand{\colim}{%
  \mathop{\mathpalette\colim@{\rightarrowfill@\textstyle}}\nmlimits@
}
\makeatother

\title[Equivariant bordism of surfaces]{Oriented and unitary equivariant bordism of surfaces}

\author{Andr\'es Angel}
	\address{Departmento de Matematicas\\
		Universidad de los Andes\\ 
		Carrera 1 N. 18A -12, Bogot\'a, Colombia}
	\email{ja.angel908@uniandes.edu.co}

\author{Eric Samperton}
	\address{Department of Mathematics\\
		University of Illinois at Urbana-Champaign\\ 
		1409 W. Green Street (MC-382), Urbana, IL 61801, USA}
	\email{smprtn@illinois.edu}
	\urladdr{https://smprtn.pages.math.illinois.edu}

\author{Carlos Segovia}
	\address{Instituto de Matem\'aticas\\
		UNAM Unidad Oaxaca\\
		Antonio de Le\'on \#2, altos, Col. Centro\\
		Oaxaca de Ju\'arez, CP. 68000 \\
		M\'exico}
	\email{csegovia@matem.unam.mx}
\author{Bernardo Uribe}
	\address{Departamento de Matem\'aticas y Estad\'istica\\
		Universidad del Norte\\ 
		Km. 5 via Puerto Colombia, Barranquilla, Colombia}
	\address{Max Planck Institut f\"ur Mathematik\\
		Vivatsgasse 7\\ 
		Bonn 53115, Germany}
	\email{bjongbloed@uninorte.edu.co, uribe@mpim-bonn.mpg.de}
	\urladdr{https://sites.google.com/site/bernardouribejongbloed/}


\thanks{The first author was partially supported by the grant (\#INV-2019-84-1860) from the Fondo de Investigaciones de la Facultad de Ciencias de la Universidad de los Andes. The second author is supported by NSF grant DMS \#2038020. The third author is supported by c\'atedras CONACYT, Convocatoria PAEP-2018 and Proyecto CONACYT ciencias b\'asicas 2016, No. 284621. The fourth author acknowledges and thanks the continuous support of the Alexander Von Humboldt Foundation and of CONACYT through project 
CB-2017-2018-A1-S-30345-F-3125. We are indebted to Prof. Peter Landweber for reading
 earlier versions of this work and for suggesting changes which have improved the paper. Thank you Prof. Landweber. }

\begin{document}

\begin{abstract}  
Fix a finite group $G$.  We study $\Omega^{SO,G}_2$ and $\Omega^{U,G}_2$, the unitary and oriented bordism groups of smooth $G$-equivariant compact surfaces, respectively, and
we calculate them explicitly. Their ranks are determined by the possible representations around fixed points, while their torsion subgroups are isomorphic to the direct sum of the Bogomolov multipliers of the Weyl groups of representatives of conjugacy classes of all subgroups of $G$.
We present an alternative proof of the fact that surfaces with free actions  which induce non-trivial
 elements in the Bogomolov multiplier of the group cannot equivariantly bound.  This result permits us to show that the 2-dimensional SK-groups
 (Schneiden und Kleben, or ``cut and paste") of the classifying spaces of a finite group can be understood in terms of  the bordism group of free
 equivariant surfaces modulo the ones that bound arbitrary actions.
\end{abstract}

\maketitle

\section{Introduction}
Equivariant bordism groups have been a subject of ongoing research since the 1960s.  Conner, Floyd, Landweber, Stong, Smith and tom Dieck, among others, laid the foundations for the extraordinary homology and cohomology theories obtained from equivariant bordism, and found many interesting properties of these groups.  Given a finite group $G$, a particularly important problem is the explicit calculation of the oriented and complex $G$-equivariant bordism groups of a point, since they provide the coefficients for the theories.  This turns out to be a complicated task.

Explicit calculations of the equivariant bordism groups for finite abelian groups \cite{Stong-complex, Landweber-cyclic,  Ossa} led some to expect that, at least in the unitary case, equivariant
bordism groups are always a free module over the unitary bordism ring for any finite group $G$ (see \cite[Page 1]{Rowlett-metacyclic}, \cite[Ch. XXVIII.5]{May-book}, \cite[Conj. 1.2]{GreenleesMay}). This belief was confirmed for general abelian groups \cite{Loffler} \cite[Ch. XXVIII, Thm. 5.1]{May-book} and for metacyclic groups \cite{Rowlett-metacyclic}, and therefore it was conjectured that for any finite group this was the case. This conjecture remained dormant for some years and it was recalled by the fourth author in his 2018 ICM Lecture \cite{Uribe-evenness} where he coined it "The evenness conjecture in equivariant unitary bordism". 

When the evenness conjecture holds true for a group $G$, it implies that the $G$-equivariant unitary bordism ring is torsion-free.
In particular, any unitary manifold with a free action of a finite group that generates a torsion
class in the unitary bordism group of free actions would bound equivariantly.
This has always been the first step for proving the evenness conjecture, namely, to construct explicit equivariant manifolds whose boundaries are the desired generators of the equivariant unitary bordism groups of free actions.

In the case of surfaces, the evenness conjecture would imply that all oriented surfaces with orientation preserving free actions bound equivariantly (note that if an oriented surface with orientation preserving free action does not  bound equivariantly,
then the class of the difference of this surface with $G$-times the quotient surface induces a non-trivial torsion class in the reduced $G$-equivariant unitary bordism group).
The third author in joint work with Dominguez \cite{Segovia} showed that indeed this was the case for abelian, dihedral, symmetric and alternating groups. Nevertheless, it fails to be true in general. It has been
recently shown by the second author that there is an obstruction class for an oriented surface with an orientation preserving free action to bound equivariantly \cite{Samperton-schur, Samperton-free}, and this obstruction class
lies in the Bogomolov multiplier of the group \cite{Bogomolov, Kunyavskii}
(the Bogomolov multiplier of a finite group consists of the classes of the Schur multiplier $H^2(G,\IC^*)$
that vanish once restricted to any abelian subgroup; the homological version of the Bogomolov multiplier
is the quotient of the second integral homology of the group by the classes generated by two dimensional tori \cite{Moravec}).
 This result implies that indeed there are torsion classes in the
equivariant unitary bordism groups and therefore that the evenness conjecture in equivariant unitary bordism is false in general. The evenness conjecture might then be
restated instead as a classification question, namely: {\it Which finite groups satisfy the evenness conjecture in equivariant unitary bordism?}

In this work we focus on the calculation of the oriented and the unitary $G$-equivariant bordism groups for compact surfaces. We use the 
fixed point construction methods developed in \cite{Rowlett} to determine the rank of the equivariant bordism groups, and
then, we use the explicit generators of the equivariant bordism groups for adjacent families in dimension 3 in order to determine
which equivariant surfaces bound. In Theorem \ref{theorem free actions that bound} we present a generalization to all finite groups of the result
shown by the second author in \cite{Samperton-schur}, which states that  the obstruction class for equivariantly
bounding an oriented surface with free action is the element in the Bogomolov multiplier of the group that the surface defines.
The Conner-Floyd spectral sequence will then allow us to determine the torsion
group in the equivariant bordism group of surfaces.
Our main result is:
\begin{theorem*} {\bf \ref{theorem torsion subgroup bordism}.}
Let $G$ be a finite group and  $\Tor_\IZ (\Omega^{G}_2)$ the torsion subgroup of the unitary
or oriented $G$-equivariant bordism of surfaces $\Omega^{G}_2$. Then there is a canonical isomorphism
\begin{equation*}
  \bigoplus_{(K)}  \tilde{B}_0(W_K) \cong  \Tor_\IZ (\Omega^{G}_2)
\end{equation*}
where $(K)$ runs over all conjugacy classes of subgroups of $G$, $W_K=N_GK/K$ and $\tilde{B}_0(W_K)$ is 
the homology version of the Bogomolov multiplier of the group $W_K$.
\end{theorem*}
With the torsion group in hand, we describe explicitly in Theorem \ref{theorem decomposition equivariant bordism of surfaces} the $G$-equivariant bordism groups of surfaces,
unitary and oriented.


Since there are infinitely many groups with non-trivial Bogomolov multipliers, we conclude that there are infinitely many groups which do not satisfy the evenness conjecture in equivariant unitary bordism. On the other hand, there are also infinitely many groups $G$ whose $G$-equivariant unitary bordism group of surfaces is a free abelian group, thus implying that these groups may still satisfy the evenness conjecture for equivariant unitary bordism.

We use our previous calculations to interpret  which equivariant surfaces bound in terms of the $SK$ relation (cutting and pasting 
from the German {\it Schneiden und Kleben}).  
The study of invariants under cutting and pasting started with the characterization by J\"anich \cite{Jaenich1,Jaenich2}
 of invariants with the additive properties of the Euler characteristic and the signature, and it was  further developed with the
 introduction of the $SK$-groups of a space in the book \cite{KKNO}. The $SK$-groups of a space can be understood as the
 groups of equivalence classes of manifolds with continuous maps to the space subject to the equivalence relation given by
 {\it cutting and pasting}.  The 2-dimensional $SK$-groups of $BG$ can be understood in terms of cutting and pasting
 surfaces with free $G$-actions. The $SK$-groups of $BG$ were studied in \cite{KKNO} and were identified
 in \cite[Thm. 2]{Neumann} with the second integral homology group of $BG$ modulo the toral classes (as far as we know 
this is the first reference where the homological Bogomolov multiplier appears).

We finish this work with the study of two explicit groups of order 64 and 243 respectively whose Bogomolov multipliers are non-trivial. 
We sketch why both groups possess non-trivial Bogomolov multipliers and we give explicit
homomorphisms from the fundamental group of a genus 2 surface to both groups that define the desired surfaces with free actions
that do not bound equivariantly. These constructions allow us to give explicit generators for the torsion subgroup
of the equivariant unitary bordism groups for both groups.

\section{Preliminaries}
\label{s:prelim}

\subsection{Equivariant bordism}
\label{ss:bordism}
Let $G$ be a finite group and consider compact manifolds endowed with smooth actions of the group $G$ preserving
 either the orientation or the unitary (tangentially stable almost complex) structure.
 
 Recall that a 
 tangentially stable almost complex $G$-structure over the $G$-manifold $M$ consists of a $G$-equivariant complex vector bundle $\xi$ over $M$,
 such that $TM \oplus \mathbb{R}^k \cong \xi$ as $G$-equivariant real vector bundles and $k$ is some natural number; here $G$ acts trivially on the stabilized part $\mathbb{R}^k$.
 Two tangentially stable almost complex structures are identified if they become isomorphic as complex vector bundles after stabilization with further
 $G$-trivial $\mathbb{C}$ summands. 
 
 With this definition at hand, if $K$ is a subgroup of $G$, then the fixed points set $M^K$ is endowed with a  canonical tangential stable almost complex 
 $W_K$-structure with $W_K:=N_GK/K$. This follows from the isomorphism of $W_K$-equivariant real bundles 
 \begin{equation}
 \xi^K \cong (TM \oplus \mathbb{R}^k)^K \cong (TM|_{M^K})^K \oplus \mathbb{R}^k = T(M^K) \oplus \mathbb{R}^k
 \end{equation}
 and the fact that $\xi^K$ becomes a $W_K$-equivariant complex vector bundle over $M^K$.
 
 Now, as $N_GK$-equivariant real vector bundles we have the isomorphism
 \begin{equation} \label{complex structure normal bundle}
 \xi|_{M^K} \cong TM|_{M^K} \oplus \mathbb{R}^k \cong T(M^K) \oplus \nu(M^K,M) \oplus \mathbb{R}^k \cong \xi^K \oplus \nu(M^K,M) ,
 \end{equation}
where $\nu(M^K,M)$ denotes the normal bundle of the embedding $M^K \hookrightarrow M$.
Since both $ \xi|_{M^K}$ and $\xi^K$ are $N_GK$-equivariant complex vector bundles over $M^K$,
then the normal bundle   $\nu(M^K,M)$ is naturally endowed with the structure of a $N_GK$-equivariant
complex vector bundle. The fact that the normal bundles of the fixed points $M^K$ are endowed with complex structures plays an important role in the study of tangentially stable almost complex $G$-structures.

Tangentially stable almost complex $G$-structures are also called $G$-equivariant unitary structures, and
the equivalence classes of manifolds under the bordism relation in the realm of $G$-equivariant unitary structures is called the $G$-equivariant unitary bordism group.

Following the notation of Stong \cite{Stong-complex}, denote by $\Omega^G_*$ either the bordism ring $\Omega^{SO,G}_*$
of $G$-equivariant oriented manifolds or the bordism ring $\Omega^{U,G}_*$ of $G$-equivariant unitary 
(tangentially stable almost complex) manifolds. 
Whenever the upper script $SO$ or $U$ is not specified, it means that the construction and results apply to both homology theories.

The explicit definitions of both unitary and oriented equivariant bordism rings can be found in 
 \cite[section 2]{Stong-complex}, and the properties of the tangentially stable almost complex manifolds 
 which define the unitary equivariant bordism groups, including the ones presented above, can be read in \cite[XXVIII, section 3]{May-book}, \cite[section 2]{Hanke} and \cite[section 5]{AngelGomezUribe}.

\subsection{Equivariant bordism for families}
\label{ss:families}
The study of the equivariant bordism groups led Conner and Floyd to restrict their attention to manifolds
with prescribed isotropy groups \cite{ConnerFloyd-book, ConnerFloyd-Odd}. The allowed isotropy groups are therefore organized
in families of subgroups of $G$ which are closed under conjugation and under taking subgroups. For any such family
of subgroups $\calf$ there is a classifying $G$-space $E \calf$ for actions whose isotropy groups lie on $\calf$. This
$G$-space is characterized by its properties on fixed points sets, namely, the fixed point set $E \calf^H$
is contractible whenever $H \in \calf$ and empty otherwise. The construction of $E \calf$ can be carried out in such a way
that an inclusion of families $\calf ' \subset \calf$ induces a $G$-cofibration $E \calf' \to E \calf$ \cite[section 1.6]{tomDieck-transformation}.

The equivariant bordism groups $\Omega_*^G\{\calf, \calf'\}$ for a pair of families $\calf ' \subset \calf$ are 
the bordism groups of $G$-equivariant compact manifolds with boundary $(M, \partial M)$ such that
the isotropy groups of $M$ lie in $\calf$ and the isotropy groups of its boundary $\partial M$ lie in $\calf'$. Following \cite[p.310]{tomDieck-Orbit-I}
one may define the bordism of groups for a pair of $G$-spaces $(X,A)$ and a pair of families as follows
\begin{equation}
\Omega^G_*\{\calf, \calf'\}(X,A) := \Omega^G_*(X \times E \calf, X \times E\calf' \cup A \times E \calf);
\end{equation}
or equivalently,  using a more geometrical description \cite{Stong-complex}.

\subsection{Long exact sequence for families}
\label{ss:long}
Whenever three families are related by the inclusions $ \calf'' \subset \calf ' \subset \calf$ there is induced a
 long exact sequence in bordism \cite[Thm. 5.1]{ConnerFloyd-Odd}
\begin{equation}
\cdots \to \Omega_*^G\{\calf', \calf''\} \to \Omega_*^G\{\calf, \calf''\} \to \Omega_*^G\{\calf, \calf'\} \stackrel{\partial}{\to} \Omega_{*-1}^G\{\calf', \calf''\} \to \cdots .
\end{equation}

\subsection{Conner-Floyd spectral sequence}
\label{ss:CFSS}
More generally, associated to the families $\calf_0 \subset \calf_1 \subset \cdots \subset \calf_k=\calf$ there is a spectral sequence
converging to $ \Omega^G_n\{\calf\}$, whose filtration is 
\begin{equation}
F_p \Omega^G_n\{\calf\} := \Image( \Omega^G_n\{\calf_p\} \to  \Omega^G_n\{\calf\}).
\end{equation}
This spectral sequence is usually called the {\it Conner-Floyd spectral sequence}, its first page is given by
\begin{equation}
E^1_{p,q} \cong \Omega^G_{p+q}\{\calf_p, \calf_{p-1}\},
\end{equation}
and the differentials are induced by the boundary maps. The first page of this spectral sequence
might be difficult to calculate, but whenever the pair of families $\calf_{p-1} \subset \calf_p$ are adjacent (see below for the definition), fixed point
methods together with the classification of the normal bundles can make them computable in terms of non-equivariant bordism groups.

\subsection{Equivariant bordism for adjacent families}
\label{ss:adjacent}
A pair of families $\calf' \subset \calf$ are called {\it adjacent} whenever they differ by the conjugacy class $(K)$
of a subgroup $K$, in other words $\calf - \calf' = (K)$. A manifold $(M, \partial M)$ in $\Omega^G_n\{\calf, \calf'\}$
is cobordant to the $G$-equivariant tubular neighborhood of the fixed point set of all the subgroups of $G$ conjugate to $K$ (all isotropy groups in the complement of the tubular nighborhood belong to $\calf'$; the explicit bordism 
can be found in \cite[Lem. 5.2]{ConnerFloyd-Odd}). The fixed points $M^K$ of $K$
become a free $W_K:=N_GK/K$ space and the $G$-equivariant tubular neighborhood can be reconstructed from a specific
$W_K$-equivariant twisted bundle over $M^K$ by extending the $N_GK$ space to a $G$ space. Hence, if $M^K$ is of dimension $n-k$
and $M^K/W_K$ is connected, its tubular neighborhood can be recovered from a map $M^K \to C_{N_GK,K}(k)$
where $C_{N_GK,K}(k)$ is a $W_K$-space which classifies the $N_GK$-equivariant tubular neighborhoods of rank $k$ around $K$-fixed points \cite[Eqn. (2.5)]{Uribe-evenness}.
In the unitary case there is a 
 decomposition in terms of non-equivariant unitary bordism groups \cite[Thm. 2.8]{Uribe-evenness} 
\begin{equation}
\Omega^{U,G}_n\{\calf, \calf'\} := \bigoplus_{2k \leq n}\Omega^{U}_{n-2k}\left(C_{N_GK,K}(k) \times_{W_K}EW_K\right),
\label{decomposition adjacent families}
\end{equation}
and a similar one in the case of oriented bordisms \cite[Thm. 2.11]{Angel-cobordism}
This localization theorem will become very useful in what follows once we apply it for the study of the equivariant bordism groups of surfaces.

\subsection{$G$ fixed points}
\label{ss:Gfixed}
For every subgroup $K$ of $G$ denote by  $\cala K$  the family of all subgroups of $K$ and its conjugates in $G$, and denote
by $\calp K$ the family $\cala K - (K)$ of all \emph{proper} subgroups of $K$, and its conjugates in $G$.
 The localization map at the fixed points of the whole group action 
\begin{equation}
\Omega_*^G \to \Omega_*^G\{ \cala G, \calp G\} 
\end{equation}
together with the decomposition into non-equivariant bordisms groups presented in \eqref{decomposition adjacent families},
has been a powerful tool for determining the equivariant bordism groups for abelian groups (see for instance \cite{tomDieck-Bordism-G, Loffler, Hanke}). In this particular case, the bordism groups $\Omega_*^G\{ \cala G, \calp G\} $ are isomorphic
to the non-equivariant bordism groups of products of complex Grassmannians in the unitary case, and of products
of real, complex and quaternionic Grassmannians in the oriented case.

\subsection{Rowlett spectral sequence}
\label{ss:RSS}
We still need another spectral sequence suited for understanding the equivariant bordism groups of pairs of families. 
This spectral sequence was constructed by Rowlett in \cite[Prop. 2.1]{Rowlett} and therefore its name.
Consider a pair of families $\calf' \subset \calf$ that are also families of subgroups of the normal subgroup $K$ of $G$ and $(M, \partial M)$ in $\Omega^G_n\{\calf, \calf'\}$.
Then it is easy to see that the classifying map $M/K \to EW_K$ of the free $W_K=G/K$ action of the quotient induces an isomorphism
of bordism groups $\Omega_*^G\{\calf, \calf'\} \stackrel{\cong}{\to} \Omega_*^G\{\calf, \calf'\}(EW_K)$
by mapping $M$ to the composition $M \to M/K \to EW_K$; the inverse is simply induced by the map $EW_K \to *$.
The space $EW_K$ can be constructed as a CW-complex whose $n$-skeleton $(EW_K)^n$ is constructed from $(EW_K)^{n-1}$
by attaching a finite number of copies of $W_K \times B^n$ with $W_K$ acting trivially on the $n$-dimensional balls. One 
may filter $\Omega_*^G\{\calf, \calf'\}(EW_K)$ by the images under the inclusion of the skeletons
$\Omega_*^G\{\calf, \calf'\}((EW_K)^n)$ and therefore one obtains a spectral sequence converging to 
$\Omega_*^G\{\calf, \calf'\}$ whose first page becomes:
\begin{align}
E^1_{p,q} & \cong \Omega_{p+q}^G\{\calf, \calf'\}((EW_K)^{p}, (EW_K)^{p-1}) \nonumber\\
& \cong H_p((EW_K)^{p}, (EW_K)^{p-1}) \otimes_{W_K}
\Omega_{q}^K\{\calf, \calf'\},
\end{align}
and whose second page is:
\begin{equation}
E^2_{p,q} \cong H_p(W_K, \Omega_{q}^{K}\{\calf, \calf'\}),
\end{equation}
where the action of an element of $W_K$ on an $K$-manifold $M$ consists of the same manifold $M$
endowed with the conjugate $K$-action. The zero-th column consists of the $W_K$-coinvariants
\begin{equation}
E_{0,q}^2 \cong \left( \Omega_{q}^{K}\{\calf, \calf'\} \right)_{W_K},
\end{equation}
and the edge homomorphism 
\begin{equation}
\Omega_{q}^{K}\{\calf, \calf'\} \cong E_{0,q}^1 \to E_{0,q}^2 \to E_{0,q}^\infty \to \Omega_{q}^{G}\{\calf, \calf'\}
\end{equation}
is simply the extension homomorphism factorizing through the coinvariants
\begin{equation} \label{edge homomorphism}
\Omega_{q}^{K}\{\calf, \calf'\} \to (\Omega_{q}^{K}\{\calf, \calf'\})_{W_K} \to  \Omega_{q}^{G}\{\calf, \calf'\}, \ \ \ M \mapsto M \times_{K} G.
\end{equation}

In characteristic zero the spectral sequence collapses on the zero-th column of the second page. Since in characteristic zero
the invariants and the coinvariants are isomorphic we conclude that the extension homomorphism induces an isomorphism
\begin{equation}
\Omega_{*}^{K}\{\calf, \calf'\}^{W_K} \otimes \IQ \stackrel{\cong}{\to} \Omega_{*}^{G}\{\calf, \calf'\} \otimes \IQ.
\label{Rowlett spectral sequence char 0}
\end{equation}

In order to find the torsion classes in $\Omega_{*}^{G}$ we will construct the inverse map of the isomorphism
\eqref{Rowlett spectral sequence char 0} for every pair of adjacent families of groups. This map will be simply given
by the localization at fixed points and will be the subject of the next section.

\section{Localization at fixed points}
\label{s:localization}
For every subgroup $K$ of $G$ let us define the fixed-point homomorphism
\begin{equation}
f_K \circ r^G_K: \Omega_{*}^{G} \to \Omega_{*}^{K}\{\cala K,\calp K\}
\end{equation} 
as the composition of the restriction homomorphism
$
r^G_K: \Omega_{*}^{G} \to \Omega_{*}^{K}
$
together with the localization at $K$-fixed points
\begin{align} \label{localization K-fied points}
f_K: \Omega_{*}^{K} \to \Omega_{*}^{K}\{\cala K,\calp K\}.
\end{align}
The composition $f_K \circ r^G_K$ takes a $G$-manifold
and maps it to the tubular neighborhood $N$ of the $K$-invariant points $M^K$. Since on the complement of $N$ in $M$
there are no points with isotropy $K$, the tubular neighborhood $N$ and $M$ become cobordant in $\Omega_{*}^{K}\{\cala K,\calp K\}$ \cite[Lem. 5.2]{ConnerFloyd-Odd}. Since $N_GK$ acts on the normal bundle $N$ of $M^K$,  the localization at $K$-fixed points
lands in the $W_K$ fixed submodule. Therefore the fixed-point homomorphism becomes
\begin{equation}
f_K \circ r^G_K: \Omega_{*}^{G} \to \Omega_{*}^{K}\{\cala K,\calp K\}^{W_K}.
\end{equation}

Also, for every pair of families of subgroups in $G$, we have the localized fixed-point homomorphism
\begin{equation} \label{fixed-point homomorphism}
\phi_*: \Omega_{*}^{G}\{\calf, \calf'\} \to \bigoplus_{(K) \subset \calf - \calf'} \Omega_{*}^{K}\{\cala K,\calp K\}^{W_K}.
\end{equation}
This homomorphism applied to the pair of adjacent families $\{\cala K, \calp K\}$, composed with the 
edge homomorphism of the Rowlett spectral sequence \eqref{edge homomorphism},
gives us the following maps
\begin{equation}
\Omega^K_*\{\cala K, \calp K\}_{W_A} \to \Omega^G_*\{\cala K, \calp K\} \stackrel{\phi}{\to} \Omega^K_*\{\cala K, \calp K\}^{W_A}.
\end{equation}
In characteristic zero this composition is an isomorphism and therefore we obtain the isomorphism
\begin{equation}
\phi_*: \Omega^G_*\{\cala K, \calp K\} \otimes \IQ \stackrel{\cong}{\to} \Omega^K_*\{\cala K, \calp K\}^{W_A}\otimes \IQ,
\end{equation}
which becomes the inverse of the map in \eqref{Rowlett spectral sequence char 0} for adjacent families.

Applying the Conner-Floyd spectral sequence we see that the fixed-point homomorphism of \eqref{fixed-point homomorphism}
in characteristic zero becomes an isomorphism, and therefore we quote:

\begin{theorem}\cite[Thm. 1.1]{Rowlett}
The fixed-point homomorphism in characteristic zero is an isomorphism
\begin{equation}
\phi_* \otimes \IQ: \Omega_{*}^{G} \otimes \IQ \stackrel{\cong}{\to} \bigoplus_{(K)} \Omega_{*}^{K}\{\cala K,\calp K\}^{W_K} \otimes \IQ.
\end{equation}
\label{th:Rowlett}
\end{theorem}

We would like to remark that the rational isomorphism obtained in 
Theorem \ref{th:Rowlett} by localizing on fixed points
holds in general for any rational $G$-equivariant
homology theory whose coefficients form a rational $G$-Mackey functor \cite[Thm A.16]{GreenleesMay-Tate}, \cite[Cor. 3.4.28]{Schwede}.

\subsection{Kernel of fixed-point homomorphism}
\label{ss:kernel}
In the unitary case, the equivariant bordism group $\Omega^{U,K}_*\{\cala K,\calp K\}$ is isomorphic to the unitary bordism group
of a disjoint union of products of complex Grassmannians \cite[Thm. 2.8]{Uribe-evenness}. Therefore, the group $\Omega_{*}^{U,K}\{\cala K,\calp K\}$ is a free
$\Omega^U_*$-module on even dimensional generators. Hence, by Theorem \ref{th:Rowlett}, we obtain the following result:

\begin{lemma}
The group of torsion elements in $\Omega^{U,G}_*$ is isomorphic to the kernel of the fixed-point homomorphism $\phi$ of \eqref{fixed-point homomorphism}:
\begin{equation}
\Tor_\IZ (\Omega^{U,G}_*) = \Ker(\phi^U_*).
\end{equation}
\label{l:torsion}
\end{lemma}

Whenever a group $G$ satisfies the evenness conjecture in equivariant unitary bordism then the fixed-point homomorphism $\phi_*^U$ is automatically a monomorphism. This is the case
for abelian \cite{Loffler} and metacyclic  \cite{Rowlett-metacyclic} groups. In the next section we will show that there are groups $G$ such that the kernel of 
the fixed-point homomorphism is not trivial in dimension 2, thus defining torsion elements in $\Omega^{U,G}_2$. This 
fact refutes the evenness conjecture in the general case.

In the oriented case there are many torsion classes in the bordism ring $\Omega^{SO}_*$, all of order 2 \cite{Wall, Stong-notes}. Therefore we will
be mainly interested in the torsion classes of the equivariant bordism group $\Omega^{SO,G}_*$ which are
trivial under the fixed-point homomorphism $\phi^{SO}_*$.

A very interesting and more general question associated to the equivariant oriented case would be the following:

{\it Are there $G$-equivariant oriented manifolds whose bordism class vanishes under the fixed-point
homomorphism $\phi^{SO}$ which do not bound equivariantly?}

In the next section  we answer this question for dimension 2. The 3-dimensional case (with its interesting application to Chern-Simons theory)
 remain open for the interested reader.

Note that the equivariant bordism group $\Omega^{SO,K}_*\{\cala K,\calp K\}$ is in general more difficult to
calculate than the unitary one. On the one hand the fixed point set $M^K$ need not be orientable, and on the other,
the normal bundles are classified by products of real, complex and quaternionic Grassmannians. 

Since we are mainly interested in the 2 and the 3-dimensional bordism groups, we know that all fixed points are of real codimension 0, 2 in the unitary case because the normal bundles are endowed with a complex structure, see equation \eqref{complex structure normal bundle}, and 0, 2 or 3 in the oriented case, because there are no 1-dimensional real representations preserving the orientation. Here the real codimension of the fixed points
match the real dimension of the representation of the respective isotropy group.

In the case that the fixed points are of real codimension 2, the normal bundle is of complex dimension 1 in the unitary case and
of real dimension 2 in the oriented case. 
Since the 2-dimensional oriented representations can be parametrized by the 1-dimensional complex 
representations, we may denote by  
$ \Irr_{\mathbb{C}}^1(K)$
 the set of 1-dimensional non-trivial irreducible complex 
representations of the group $K$. The complex conjugation map on $ \Irr_{\mathbb{C}}^1(K)$
acts freely on the representations of complex type $ \Irr_{\mathbb{C}}^1(K)_\mathbb{C}$ 
and acts trivially on the representations of real type $ \Irr_{\mathbb{C}}^1(K)_\mathbb{R}$.
Denote by $ \Irr_{\mathbb{C}}^1(K)_\mathbb{C}/\mathrm{conj}$ the quotient of representations
of complex type by complex conjugation and denote by $\Irr_{\mathbb{R},SO}^3(K)$
the set of 3-dimensional irreducible real representations of $K$ in the category of oriented representations.

\begin{proposition} \label{p:23AGPG}
Let $K$ be a finite group. Then the relative oriented equivariant bordism groups are 
\begin{align} \label{SO,2,K}
\Omega^{SO,K}_2\{\cala K,\calp K\} & = \left( \bigoplus_{ \Irr_{\mathbb{C}}^1(K)_\mathbb{C}/\mathrm{conj}} \mathbb{Z} \right) \oplus \left( \bigoplus_{\Irr_{\mathbb{C}}^1(K)_\mathbb{R}} \mathbb{Z}/2 \right), \\
\label{SO,K,3}
\Omega^{SO,K}_3\{\cala K,\calp K\} & = \bigoplus_{\Irr_{\mathbb{R},SO}^3(K)} \mathbb{Z}/2,
\end{align}
and the relative equivariant unitary bordism groups are
\begin{align}
\label{U,2,K} \Omega^{U,K}_2\{\cala K,\calp K\} & = \Omega_2^U \oplus \bigoplus_{\Irr_{\mathbb{C}}^1(K)} \mathbb{Z},  \\
\label{U,3,K} \Omega^{U,K}_3\{\cala K,\calp K\} & =0.
\end{align}
\end{proposition}

\begin{proof}
Let us start with the relative oriented equivariant bordism groups. Any manifold $M$ in $\Omega^{SO.K}_*\{\cala K,\calp K\}$ is equivalent in the bordism group to the normal bundle $N$ around the fixed point set $M^K$ \cite[Lem. 5.2]{ConnerFloyd-Odd}. 
Whenever $M$ is connected, of dimension 2, and $M \neq M^K$, this normal bundle is classified by a map
\begin{align} \label{classifying map normal bundle}
M^K \to \bigsqcup_{\Irr_\mathbb{C}^1(K)} BU(1)
\end{align} where the $K$ action on the bundle around the point is encoded by the irreducible representation (here we are using that $SO(2) \cong U(1)$). Note that whenever $V$ is a non-trivial 1-dimensional complex representation,
 the unit ball $B(\mathbb{R} \oplus V)$ bounds the union of $B(V)$ and $B(\overline{V})$ where
 $\overline{V}$ denotes the representation $V$ with reverse orientation. This implies that in the relative oriented bordism group
 $\Omega^{K, SO}_2\{\cala K,\calp K\}$ we have the equation $B(V)+B(\overline{V})=0$. Hence whenever $V$ is of complex type, and therefore $V$ 
 is not isomorphic to $\overline{V}$, the relative oriented bordism group
 $\Omega^{K, SO}_2\{\cala K,\calp K\}$ counts the difference between the number of $K$-fixed points with normal bundle isomorphic to $\overline{V}$
 and the number of $K$-fixed points with normal bundle isomorphic to $V$; these are the integral invariants. If $V$ is of real type and hence $V$ is isomorphic to $\overline{V}$, the ball $B(\mathbb{R} \oplus V)$
 bounds twice $B(V)$ and the relative oriented bordism group
 $\Omega^{K, SO}_2\{\cala K,\calp K\}$ counts the parity of the number of points with normal bundle isomorphic to $V$; these are the $\mathbb{Z}/2$ invariants. This argument proves equation \eqref{SO,2,K}.
 
 For the 3-dimensional case, the codimension 2 fixed points become circles, and since $\Omega^{SO}_1(BU(1))=0$,
 we conclude that we only need to focus our attention on the isolated points of the $K$ action. Around each isolated fixed point of the action we
 obtain a 3-dimensional real and oriented representation $V$ of $K$. This representation is irreducible in the category of oriented representations
 even though it may be not irreducible as a real representation. Note that the splitting of the representation as the product of two non-oriented representations 
 implies that one must be a sign representation and the other must factor through a dihedral representation in $O(2)$. Hence the product of these two
 representations will be equivalent to a representation that factors through an oriented dihedral representation in $SO(3)$ which is irreducible
 in the category of oriented representations. Now, the unit ball
 $B(\mathbb{R} \oplus V)$ bounds twice $B(V)$ because $V$ and $\overline{V}$ are isomorphic. Therefore we can conclude that
 the isomorphism of equation \eqref{SO,K,3} counts the parity of the number of fixed points of $K$ with the prescribed representation on its normal bundle.

 The relative unitary bordism groups are much simpler. The 3-dimensional case of equation
 \eqref{U,3,K} is trivial because both $\Omega^U_3$ and $\Omega_1^U(BU(1))$ are trivial. The 2-dimensional
 case of \eqref{U,2,K} detects half of the fisrt Chern number of the surface whenever the action is trivial,
 and it counts the number of fixed points with prescribed representation on their normal bundle. Here
 we are using that the isomorphism $\Omega^U_2 \stackrel{\cong}{\to} \mathbb{Z}$ is given by
 the assignment $[\Sigma] \mapsto \frac{c_1(\Sigma)}{2}$ where $c_1(\Sigma)$ is the first Chern number of the surface.

\end{proof}

As a consequence of the previous result we see that the 2-dimensional bordism classes of interest
have no isolated fixed points for any subgroup $K$ of $G$.

\begin{corollary}
The torsion subgroups of both unitary and oriented equivariant bordism of surfaces, are respectively isomorphic
to the kernel of the associated fixed point homomorphism,
\begin{equation}
\Tor_\IZ (\Omega^{U,G}_2) = \Ker(\phi^U_2), \ \ \Tor_\IZ (\Omega^{SO,G}_2) = \Ker(\phi^{SO}_2).
\end{equation}
Therefore the equivariant bordism groups $\Ker(\phi_2^{SO})$ and $\Ker(\phi_2^{U})$  are generated
by $G$-surfaces without isolated $K$-fixed points for any  subgroup $K$ of $G$; in the unitary case it is moreover required that 
the surfaces have trivial first Chern number.
\label{c:nofixed}
\end{corollary}

\begin{proof} 
In Proposition \ref{p:23AGPG} we have shown that the relative oriented and unitary bordism groups 
$\Omega^{K}_2\{\cala K,\calp K\}$ are torsion-free for all subgroups $K$ of $G$, except
in the oriented case whenever $K$ has 1-dimensional complex representations of real type;
such representations come from non-trivial elements in
$\Hom(K,\mathbb{Z}/2)$. Whenever a closed oriented surface $\Sigma$ has one $K$-fixed point
whose normal bundle has the structure of a non-trivial element in $\Hom(K,\mathbb{Z}/2)$, then
the connected component of such $K$-fixed point has an induced action of $\mathbb{Z}/2$. Since
the Euler characteristic of the connected component is even, the number
of fixed points of this  $\mathbb{Z}/2$-action must also be even. Hence the original action of $K$ on 
this connected component must have an even number of fixed points, and all of them will have isomorphic complex representation of real type on the normal bundles.

The previous argument shows that the image of the fixed point homomorphism is torsion-free in
both oriented and unitary cases. Therefore by Theorem \ref{th:Rowlett} we can conclude 
that the torsion classes are generated by $G$-equivariant manifolds without isolated
$K$-fixed points for any subgroup $K$ of $G$, and in the unitary case it is furthermore required that
the underlying surface has trivial first Chern number.

\end{proof}

The presence of the platonic groups $A_4$, $\mathfrak{S}_4$, $A_5$ or the dihedral groups $D_{2k}$ as subgroups
of a general group $G$ makes the understanding of the bordism group $\Omega_3^{SO,G}$ more interesting. We need
first a definition.

\begin{definition}
Let $M$ be a $G$-manifold $M$ (oriented or unitary). Define the {\it{ramification locus}} of the $G$-action as the space
\begin{equation}
\overline{M} := \bigcup_{K \subset G \atop K \neq \{1\}} M^K,
\end{equation}
where $M^K$ denotes the space of fixed points of the subgroup $K$.
\end{definition}

Let us start with the dihedral groups.

\begin{proposition} \label{3-d bordism of dihedral}
The equivariant bordism groups $\Omega_3^{SO,D_{2k}}$ are generated by equivariant manifolds
whose fixed-points are all of codimension 0 or empty. In particular, the fixed-point homomorphism $\phi_3^{SO}$ is trivial.
\end{proposition} 

\begin{proof}
Take $M$ a closed oriented $D_{2k}$-equivariant manifold such that $M/D_{2k}$ is connected.
Let us first assume that no element in $D_{2k}$ besides the identity acts trivially (we could always
take the induced action on $M$ of the group $D_{2k}/L$ where $L$ is the subgroup that acts trivially
and consider $M$ as a $D_{2k}/L$-equivariant manifold).
Hence we have that the ramification locus $\overline{M}$ is the union of 1-dimensional and 0-dimensional manifolds.

Whenever the fixed-point set $M^{D_{2k}}$ is non-empty, it will consist of a finite number of isolated points.
We will argue that the number of fixed-points with isomorphic normal representations is even, thus implying that
 the image of the localization map of \eqref{localization K-fied points} at $D_{2k}$-fixed points
\begin{align}
f_{D_{2k}} : \Omega_3^{SO,D_{2k}} \to \Omega_3^{SO,D_{2k}} \{ \cala D_{2k}, \calp D_{2k} \}
\end{align}
 is trivial, and moreover, that the fixed-point set $M^{D_{2k}}$ could be removed 
with an equivariant cobordism by attaching handles around pairs of fixed points with isomorphic normal representation.

 If $x$ belongs to $M^{D_{2k}}$, we claim that there
is another fixed point $x' \in M^{D_{2k}}$, such that both have isomorphic representations of $D_{2k}$ on their
normal neighborhoods. The reason for this is the following. Consider the class $[x] \in \overline{M}/D_{2k}$ on the
quotient of the ramification locus $\overline{M}$. The connected component of the fixed point set of the cyclic subgroup $\IZ/k$ around $x$
 defines a path on the quotient $ \overline{M}/D_{2k}$ with the class $[x]$ at one end. Since $ \overline{M}/D_{2k}$
is compact, the other end of this path ends at the class of the point $[x']$ where we have chosen $x'$ to be on the same
 connected component as $x$ on the fixed point set $M^{\IZ/k}$. The $D_{2k}$ representations around $x$ and $x'$ are isomorphic
because their restrictions to the group $\IZ/k$ give representations with opposite orientations.

Note that whenever $k>2$ the points $x$ and $x'$ are different. When $k=2$ it could be the case that $x=x'$, and if this were the case,
we need to notice that around $[x]$ in  $ \overline{M}/D_{2k}$ we would have a loop (the path we defined above from $x$ to $x'=x$)
and an extra path leaving from it. Following this third path from $x$, we will reach another point $x''$ which will be different from $x$.

We just have shown that the fixed points in $M^{D_{2k}}$ come in pairs with isomorphic representations. If the isomorphic
representation is $V$ and $B(V)$ denotes the unit ball in $V$, this pair of points could removed by the bordism that adds the handle $[0,1] \times B(V)$ on the normal 
neighborhoods of the pair of points.

 The previous construction could be carried out on all the fixed points of the conjugacy classes of subgroups which are of  
dihedral type, and therefore we see that $M$ is equivariantly cobordant to a manifold $M'$ whose fixed-points of its dihedral subgroups are empty.
Hence the ramification locus $\overline{M'}$ is a 1-dimensional manifold and therefore $\phi_3^{SO}([M])=\phi_3^{SO}([M'])=0$.

We could then chose as generators of $\Omega_3^{SO,D_{2k}}$ manifolds without 0-dimension and 1-dimensional fixed points.
\end{proof}

Propositions \ref{p:23AGPG} and  \ref{3-d bordism of dihedral} imply that the fixed-point homomorphism
$\phi_3^{SO}$ is trivial on subgroups isomorphic to cyclic or dihedral groups. 
Nevertheless, the fixed-point homomorpshism may be non-trivial when evaluated
on subgroups isomorphic to the platonic groups $A_4$, $\mathfrak{S}_4$ and $A_5$. 
To understand the image of $\phi_3^{SO}$ for the platonic groups we first need to define 
the {\it blow up} of a representation.

\begin{definition} \label{blow up}
Let $V$ be a finite dimensional real $G$-representation. The {\it blow up} $\gamma(V)$ of $V$ 
is the total space of the bundle of real lines $\IP(V)$ of $V$:
\begin{align}
\gamma(V) := \{ (v,L) \in V \times \IP(V) \colon v \in L \},
\end{align}
endowed with the natural $G$ action: $g \cdot (v,L):=(gv, gL)$. Denote $B(\gamma(V))$ and $S(\gamma(V))$
the unit ball and sphere bundles of $\gamma(V)$ respectively.
\end{definition}
Note that the sphere bundle of $\gamma(V)$ and the sphere of the representation $S(V)$ are canonically isomorphic
\begin{align}
\rho: S(V) \stackrel{\cong}{\to} S(\gamma(V)), \ \ v \mapsto (v,\langle v \rangle ).
\end{align}
Therefore one may glue $B(\overline{V})$, where $\overline{V}$ is $V$ with the opposite orientation, 
to $ B(\gamma(V))$ along their boundary
\begin{align} \label{definition of Y(V)}
Y(V) := B(\overline{V}) \cup_\rho B(\gamma(V)),
\end{align}
thus constructing a closed oriented $G$-manifold. 

What is interesting about the blow up is the fact that for faithful 3-dimensional oriented real representations $V$,
 the blow up $\gamma(V)$
only contains points with cyclic or dihedral isotropy groups. This is a key fact that will be used in what follows.

\begin{proposition} \label{fixed-point homomorphism on platonic}
Let $G$ be a finite subgroup of $SO(3)$. Then the fixed-point homomorphism $\phi_3^{SO}$ is only non-trivial
on subgroups isomorphic to the platonic groups $A_4$, $\mathfrak{S}_4$ and $A_5$. Moreover, its restriction
\begin{align} \label{surjectivity platonic}
\phi_3^{SO} : \Omega_3^{SO,G} \to \bigoplus_{(K) \atop K \ platonic} \Omega_{3}^{SO,K}\{\cala K,\calp K\}^{W_K}
\end{align}
is surjective. 
\end{proposition}
\begin{proof}
Let $(K)$ be a conjugacy class of subgroups of $G$ with $K$ isomorphic to any of the platonic groups
$A_4$, $\mathfrak{S}_4$ or $A_5$. Denote by $V_K$ the 3-dimensional real representation induced by the symmetries
of the respective platonic solid. Note that $V_K$ is isomorphic to the representation with the reverse orientation $\overline{V_K}$,
and therefore the closed oriented $K$-manifold $Y(V_K)$ defined in \eqref{definition of Y(V)} is diffeomorphic to 
$B(V_K) \cup_\rho B(\gamma(V_K))$. Note furthermore that $\Omega_{3}^{K}\{\cala K,\calp K\} \cong \IZ/2$ 
since $V_K$ is the only irreducible representation of dimension 3.

The localization map at $K$-fixed points of \eqref{localization K-fied points}
\begin{align}
f_K: \Omega_{3}^{SO,K} \to \Omega_{3}^{K}\{\cala K,\calp K\} \cong \IZ/2
\end{align}
maps $Y(V_K)$ to the normal bundle of its $K$-fixed points $Y(V_K)^K$. Since the blow up $\gamma(V_K)$
has no $K$-fixed points, then $Y(V_K)^K=B(V_K)^K$ and the fixed point set consists of only one point. Hence 
$f_K(Y(V_K)) = B(V_K)$  with $[B(V_K)]$ the generator of the group $\Omega_{3}^{SO, K}\{\cala K,\calp K\}$.

The commutativity of the diagram
\begin{align}
\xymatrix{
\Omega_3^{SO,K} \ar[rr]^{i_K^G} \ar[d]_{f_K} && \Omega_3^{SO,G} \ar[d]^{f_K \circ r^G_K} \\
\Omega_{3}^{SO, K}\{\cala K,\calp K\} \ar[rr]_{i_K^{N_GK}} && \Omega_{3}^{SO, K}\{\cala K,\calp K\}^{W_K},
}
\end{align}
where $i_H^L: \Omega_*^H \to \Omega_*^L$, $[M] \mapsto [L \times_H M]$ is the induction map
for the inclusion of groups $H \subset L$, implies that the manifold $f_K \circ r^G_K \left( G \times_K Y(V_K) \right)$ generates the group
$\Omega_{3}^{SO,K}\{\cala K,\calp K\}^{W_K}$.

Note that whenever  $K \subsetneqq K'$, we have $\left(G \times_K Y(V_K) \right)^{K'} = \emptyset$. Therefore we conclude that the images
under $\phi_3^{SO}$ of the $G$-manifolds $G \times_K Y(V_K)$, where $(K)$ runs over the conjugacy classes of platonic
subgroups of $G$, provides the desired surjectivity.
\end{proof}

Let us see the previous result in an example. Let $G=A_5$ and take the $A_5$-manifolds $Y(V_{A_5})$ and $A_5 \times_{A_4} Y(V_{A_4})$
in $\Omega_3^{SO,A_5}$. The image under $\phi_3^{SO}$ of these two manifolds in
\begin{align}
\Omega_{3}^{SO, A_5}\{\cala A_5,\calp A_5\} \oplus \Omega_{3}^{SO, A_4}\{\cala A_4,\calp A_4\} \cong \IZ/2 \oplus \IZ/2
\end{align}
is respectively $(1,1)$ and $(0,1)$. The surjectivity of the map \eqref{surjectivity platonic} follows.

\subsection{Surfaces without isolated fixed points for any subgroup}
Let $\calf$ be a family of subgroups in $G$ and let us  then denote by $\overline{\Omega}_2^{G}\{ \calf\}$ the subgroup of ${\Omega}_2^{G}\{ \calf\}$ generated
by manifolds without isolated $K$-fixed points for all $K \in \calf$, and whose underlying first Chern number is zero in the unitary case. Since Corollary \ref{c:nofixed} also implies that 
\begin{equation}
\overline{\Omega}_2^{G} \{ \calf\} = \Ker(\phi_2|_{{\Omega}_2^{G}\{ \calf\}}) = \Tor_\IZ (\Omega^{G}_2\{ \calf\}),
\end{equation}
we may study the properties of $\overline{\Omega}_2^{G}$ restricted to families.

\begin{lemma}
Let $\{\calf, \calf'\}$ be an adjacent pair of families differing by the conjugacy class $(K)$ of the subgroup $K \subset G$. Then,
the canonical map of bordism groups for families $ \overline{\Omega}_2^{G}\{ \calf'\} \to \overline{\Omega}_2^{G}\{\calf\}$ fits into the split exact sequence
\begin{equation}
 \overline{\Omega}_2^{G}\{ \calf'\} \to \overline{\Omega}_2^{G}\{\calf\} \to \widetilde{\Omega}_2(BW_K) \to 0,
\end{equation} 
with 
$\widetilde{\Omega}_2$ the reduced bordism groups.
\label{l:adjacentnofixed}
\end{lemma}

\begin{proof}
A generator in $\overline{\Omega}_2^{G}\{\calf\}$ not in the image of $\overline{\Omega}_2^{G}\{ \calf'\} $ is represented by a $G$-connected manifold $M$ such that the fixed point set $M^K$ is a closed non-empty surface without boundary, and such that
 there is a $G$-equivariant homomorphism $G \times_{N_GK}M^K \stackrel{\cong}{\to} M$, $[(g,m)]\mapsto gm$. 
The closed surface $M^K$ is endowed with a free action of the group $W_K$, thus producing a unique
 map up to homotopy
$M^K/W_K \to BW_K$. The induction map 
\begin{equation}
\widetilde{\Omega}_2(BW_K) \to \overline{\Omega}_2^{G}\{\calf\}, \ \ \ L \mapsto G \times_{N_GK}L
\end{equation}
produces the desired section.

For the unitary case we need only to see that the first Chern number of $M$ is zero, if and only if
the first Chern number of $M^K$ is zero, if and only if the first Chern number of $M^K/W_K$ is also zero.

Here we have used the isomorphism
\begin{equation}
 \widetilde{\Omega}_2^{U}(BW_K) \cong \Ker \left({\Omega}_2^{U}(BW_K) \to {\Omega}_2^{U}\right)
\end{equation}
where the forgetful map ${\Omega}_2^{U}(BW_K) \to {\Omega}_2^{U}$ simply takes a framed bordism $[\Sigma \to BW_K]$ and
maps it to $[\Sigma]$. The kernel consists of framed surfaces whose underlying first Chern number is zero.
In the oriented case we have $\widetilde{\Omega}_2^{SO}(BW_K) = {\Omega}_2^{SO}(BW_K)$.

\end{proof}

\section{Bounding equivariant surfaces}
\label{s:bounding}
In this section we present the main result of this work, which is the calculation  of the groups $\overline{\Omega}_2^G$.
 To address this question we bring the Conner-Floyd spectral sequence of section \ref{ss:CFSS} associated to the families of subgroups 
\begin{equation} \label{dense filtration}
\{1\}= \calf_0 \subset \calf_1 \subset \cdots \subset \calf_l=\cala G
\end{equation}
where all the pairs are adjacent, i.e. $ \calf_{j} - \calf_{j-1} = (K_j)$ for some conjugacy class of subgroups $(K_j)$,
and such that the conjugacy classes $(K_j)$ span all conjugacy classes of subgroups of $G$ (hence $l+1$ is the number
of conjugacy classes of subgroups of $G$).

We may filter the group $\overline{\Omega}_2^G$ by the subgroups
\begin{equation}
F_p \overline{\Omega}_2^G := \Image( \overline{\Omega}_2^G \{\calf_p\} \to \overline{\Omega}_2^G)
\end{equation}
whose associated graded groups are the quotients
\begin{equation}
\Gr_p \overline{\Omega}_2^G  =  F_p \overline{\Omega}_2^G / F_{p-1} \overline{\Omega}_2^G.
\end{equation}

The commutative diagram with exact rows
\begin{align}
\xymatrix@R=10pt{{\Omega}_3^G\{\cala G,\calf_{p-1}\} \ar[r] \ar[d] & \overline{\Omega}_2^G\{\calf_{p-1}\}
\ar[r] \ar[d] & \overline{\Omega}_2^G \ar@2{-} [d]\\
{\Omega}_3^G\{\cala G,\calf_p\} \ar[r] & \overline{\Omega}_2^G\{\calf_p\} \ar[r] & \overline{\Omega}_2^G,
}
\end{align}
 together with the result of Lemma \ref{l:adjacentnofixed}, implies that the following sequence is exact
\begin{equation}
 {\Omega}_3^G\{\cala G,\calf_p\} \stackrel{\partial}{\to} \widetilde{\Omega}_2(BW_{K_p})  \to \Gr_p \overline{\Omega}_2^G \to 0.
\end{equation}

We therefore need to understand the image of the boundary map
\begin{equation} \label{boundary map adjacent families}
 {\Omega}_3^G\{\cala G,\calf_p\} \stackrel{\partial}{\to} \widetilde{\Omega}_2(BW_{K_p})
\end{equation} in order to determine
the groups $\Gr_p \overline{\Omega}_2^G$.

Note that the image of the boundary map of \eqref{boundary map adjacent families} is equivalent 
to the image of the boundary map
\begin{equation} \label{boundary map W_K}
 {\Omega}_3^{W_{K_p}}\{\cala W_{K_p},\{1\}\} \stackrel{\partial}{\to} \widetilde{\Omega}_2(BW_{K_p}).
\end{equation}
This follows from the fact that the manifolds of interest will have trivial actions of the groups
in the conjugacy class $(K_p)$ and then one follows the same argument
as presented in  Lemma \ref{l:adjacentnofixed}. Therefore we obtain the following result.

\begin{lemma}
\label{l:gr}
Consider the associated graded groups $\Gr_* \overline{\Omega}_2^G$ of $\overline{\Omega}_2^G$ induced by the families of subgroups presented in \eqref{dense filtration}. Then 
\begin{equation}
\Gr_p \overline{\Omega}_2^G \cong \Coker\left({\Omega}_3^{W_{K_p}}\{\cala W_{K_p},\{1\}\} \stackrel{\partial}{\to} \widetilde{\Omega}_2(BW_{K_p})\right).
\end{equation}
\end{lemma}

Hence we need to understand which surfaces with free actions equivariantly bound.

\subsection{Bounding free actions on surfaces}
\label{ss:freebounding}

It turns out that the only free actions on surfaces that equivariantly bound are those on which the quotient surface is a torus.
This result is originally due to the second author \cite{Samperton-schur, Samperton-free} whenever
the group $G$ does not contain any subgroup isomorphic to the platonic groups $A_4, \mathfrak{S}_4,A_5$ or to the dihedral groups $D_{2k}$,
 and it was the one who triggered the investigation presented in this work.
Here we will produce an alternative proof, generalizing it for all finite groups. Let us first recall the definition of the Bogomolov multiplier of a finite group.

The cohomology group $H^2(G, \IC^*)$ determines the isomorphism classes of central $\IC^*$ group extensions of $G$, and therefore
complex irreducible projective representations of the group $G$ define elements in $H^2(G, \IC^*)$. Schur \cite{Schur} extensively studied
this cohomology group and therefore it was called the {\it Schur multiplier} of $G$ \cite{Karpilovsky}.

Bogomolov \cite{Bogomolov} defined the subgroup $B_0(G)$ of the Schur multiplier consisting of all elements which vanish
when restricted to all its abelian subgroups
\begin{equation} \label{definition Bogomolov}
B_0(G) = \bigcap_{A \subset G, \mathrm{abelian}} \Ker\left( \mathrm{res}^G_A: H^2(G, \IC^*) \to H^2(A, \IC^*) \right).
\end{equation}
The interest in this group comes among other things from a result Bogomolov proved in \cite[Thm. 3.1]{Bogomolov} which states that whenever the field of $G$-invariants $\IC[G]^G$ of the rational
field $\IC[G]$ is rational over $\IC$ then the Bogomolov multiplier of the group $G$ vanishes.

Using the fact that for finite groups $H^2(G, \IC^*) \cong H_2(G, \IZ)$, Moravec \cite{Moravec} showed that the 
Bogomolov multiplier group $B_0(G)$ is isomorphic to the group 
\begin{equation} \label{definition Bogomolov homology}
\tilde{B}_0(G) := H_2(G, \IZ) / M_0(G) 
\end{equation}
where $M_0(G)$ is the subgroup of $H_2(G, \IZ)$ generated by the images 
\begin{equation}
\Image \left(H_2(\IZ \times \IZ, \IZ) \to H_2(G, \IZ)\right) 
\end{equation}
of all homomorphisms $\IZ\times\IZ \to G$. This homology version of the Bogomolov multiplier was then used to calculate
$B_0(G)$ for several types of finite groups \cite{Moravec}.

In this homological form, the Bogomolov multiplier appeared much earlier in \cite{Neumann} in connection with $SK$-groups (cutting and pasting of manifolds) and in \cite{Oliver1980} as $SK_1$ in algebraic K-theory.

Using now the fact that there are canonical isomorphisms
\begin{equation}
\widetilde{\Omega}_2^U(BG) \stackrel{\cong}{\to} {\Omega}_2^{SO}(BG)  \stackrel{\cong}{\to} {H}_2(BG, \IZ), 
\end{equation}

we present a generalization of a result which was established by the second author in \cite{Samperton-free}. First we need a lemma.

\begin{lemma} \label{Lemma blow up}
Let $\Sigma$ be an oriented surface with free $G$-action that bounds equivariantly. Then $\Sigma$ can be extended
to an oriented $G$-manifold whose ramification locus  is a 1-dimensional manifold (all the isotropy groups are all cyclic).
\end{lemma}

\begin{proof}
Let $M$ be an oriented $G$-manifold whose boundary is the surface with free $G$-action $\Sigma$.
Take a point $x$ in the ramification locus $\overline{M}$ and denote $G_x$ its isotropy group. Since
the $G$-action is free on the boundary, the action of $G_x$ on the normal neighborhood of $x$
must induce an injective homomorphism $G_x \to SO(3)$. Hence $G_x$ must be isomorphic
to a cyclic group, a dihedral group or any of the platonic groups $A_4$, $\mathfrak{S}_4$ or $A_5$.
Whenever $G_x$ is cyclic, $x$ is a smooth point in the ramification locus $\overline{M}$, because
locally $G_x$ acts by rotations. Whenever $G_x$ is neither trivial nor cyclic, then $x$ is a singular point on the ramification locus. Simply note that the irreducible and oriented 3-dimensional representations
of the dihedral and the platonic groups have the origin as a singular point.
Therefore the obstruction for the ramification locus $\overline{M}$ to be a 1-dimensional manifold is the presence
of points whose isotropy groups are isomorphic to the dihedral or the platonic groups ($A_4$, $\mathfrak{S}_4$, $A_5$).  Our goal is to modify $M$ to build a new manifold without any such isotropies.

We briefly outline the overall strategy of our de-singularizing process.  There are three steps.
\begin{enumerate}
\item Perform the blow up construction on the normal neighborhoods of the points whose isotropies are isomorphic to either
$A_5$, $\mathfrak{S}_4$ or $A_4$; this produces a new manifold $M'$ with the same boundary as $M$ and no points with $A_5$, $\mathfrak{S}_4$ or $A_4$ isotropy.
\item  In $M'$, ``cancel" as many pairs of distinct orbits of a given dihedral isotropy type possible; our cancellation method results in a manifold $M''$ that is equivariantly cobordant to $M'$, relative to the boundary $\Sigma$.  By cancelling as many pairs as possible, we guarantee that for the action of $G$ on $M''$, a given conjugacy class of dihedral subgroup of $G$ occurs on at most one orbit in $M''$.
\item The final step is the hardest.  If $x$ is a point in $M''$ with dihedral isotropy $G_x \le G$ that is maximal, then we show that the action of $G$ on $G\cdot x$ possesses an ``involutive" element $g$ such that $y = gx \ne x$, $g^2 x =x$, $G_{y} = G_x$, and $g$ commutes with a preferred rotation $\sigma \in G_x$.  We then classify the possibilities for $\langle G_x, g\rangle$, and build an appropriate equivariant handle that desingularizes the orbit $G\cdot x$.  Inductively applying this construction to all dihedrally stabilized points, we arrive at our desired manifold $M'''$.  In fact, this is oversimplifying, and we must return to step (2) once at some point in this process, but the basic idea is as described.
\end{enumerate}

Let us expand on step (1).  Take a point $x \in M$ whose isotropy $G_x$ is isomorphic to $A_5$ (we will start with the larger isotropy first).  Let $N_x$ be a normal $G_x$-neighborhood of $x$ such that $ N_x \cap  g \cdot N_x = \emptyset$ for 
all $g \in G - G_x$, and let 
\begin{align}
\sigma: B(V_{G_x}) \stackrel{\cong}{\to} N_x
\end{align}
 be a $G_x$-equivariant diffeomorphism
with $V_{G_x}$ the faithful representation of $G_x$ around $x$.
Take $G\cdot N_x$ as a $G$-equivariant neighborhood around the orbit $G \cdot x$ and note that $\sigma$ induces
a $G$-equivariant diffeomorphism $G \times_{G_x} B(V_{G_x}) \stackrel{\cong}{\to} G \cdot N_x$.
Construct the blow up $B(\gamma(V_{G_x}))$ presented in Definition \ref{blow up} and note that the sphere
bundles are $G_x$-diffeomorphic to the boundary of $N_x$:
\begin{align}
S(\gamma(V_{G_x})) \cong S(V_{G_x}) \stackrel{\cong}{\to} \partial N_x.
\end{align}
Cut $G\cdot N_x$ from $M$ and glue $G \times_{G_x} B(\gamma(V_{G_x}))$ along the boundary
using the diffeomorphism $\sigma$. Define the new $G$-manifold
\begin{align}
M' :=\left( M - G\cdot N_x \right) \cup_{\partial( G \cdot N_x)} G \times_{G_x} B(\gamma(V_{G_x}))
\end{align}
and note that $M'$ has the same boundary as $M$, but with the property that inside $\partial N_x$
there are no more points with isotropy isomorphic to $A_5$.
Cutting and pasting the blow ups for every point with isotropy isomorphic to $A_5$, produces a 
manifold without points whose isotropy is isomorphic to $A_5$. Then a similar blow up procedure is carried out for 
points with isotropy isomorphic to $\mathfrak{S}_4$ and then to points with isotropy isomorphic to $A_4$.
The resulting manifold $M'$ has the same boundary as $M$, but it does not contain points with isotropy
isomorphic to either $A_5$, $\mathfrak{S}_4$ or $A_4$. The only isotropies that appear on $M'$ are cyclic
or dihedral groups.  This concludes step (1).  (We note for the interested reader that $M$ and $M'$ are not necessarily relatively cobordant, even though they do have the same boundary.)

For step (2), suppose $x$ and $y$ are two points in $M'$ with equal dihedral stabilizers $G_x = G_y$  such that $x$ and $y$ are not in the same $G$ orbit, but the representation of $G_x$ on a regular neighborhood of $x$ is equivalent to the representation of $G_y=G_x$ on a regular neighborhood of $y$. Call this representation $V$. Choose local charts around $x$ and $y$ such that the angle of rotations of the elements in $G_x$ agree
in both charts to the angles of rotations
 in $V$.
  Now simply attach an equivariant 4-dimensional handle of the form $[0,1] \times (G \times_{G_x} B(V))$ to the equivariant neighborhood of $\{x,y\}$. Note that the only points in this handle with non-cyclic isotropy are those in $G\cdot\{x,y\}$. 
Thus, the cobordant 3-manifold (where the open $G$-equivariant regular neighborhood of $G\cdot \{x,y\}$ is replaced with the vertical boundary of our handle) has fewer points with isotropy isomorphic to dihedral groups. 
Iterate this procedure, attaching such an equivariant handle anytime we see a pair $x$ and $y$  with the same isotropy group $G_x = G_y$,  isomorphic local representations, 
and $y \notin G\cdot x$, until there are no more such pairs.  
Call the resulting manifold $M''$.  Of course, since these handle attachments occur away from $\partial M' = \Sigma,$ $M''$ still has boundary $\Sigma$.  More precisely, $M'$ and $M''$ are equivariantly bordant relative to $\Sigma$.

Step (3) is the most involved.  Let $x$ be a point in $M''$ with dihedral isotropy $G_x \cong D_{2k}$.
Define the set $\Lambda_x$ of points $y$ in $M''$ such that $G_y=G_x$ and whose local representations on regular neighborhoods around $y$ and $x$ respectively are isomorphic.
  By construction, if $y$ is any other point in this set, then $y \in G\cdot x$. In fact, $\Lambda_x=  N_G(G_x) \cdot x$ and $\Lambda_x$ is bijective with the group $N_G(G_x)/G_x$.
    Our goal now is to build a $G$-handle that allows us to cancel the singularities in the \emph{single} orbit $G\cdot x$ with one another  in pairs in a $G$-equivariant fashion.  In particular, we will need to show that $|G \cdot x|$ is even and admits a $G$-invariant matching.

Let us first show that there is an element $g \in N_G(G_x)$ such that its
projection on $N_G(G_x)/G_x$ has order two. This $g$ will allow us to define $y:=gx$ such that $x=gy$.  We proceed by induction down the subgroup lattice of $G$ and through the different isomorphism classes
of faithful local representations of these point stabilizers. Let $x$ be any point in $M''$ whose stabilizer $G_x\cong D_{2k}$ is maximal among all dihedral point stabilizers in $M''$ (with respect to subgroup inclusion) and consider the restricted action of just $G_x$ on $M''$. 

If $k>2$, let $\sigma \in G_x$ be an element of order $k$ and take its fixed point set $(M'')^\sigma$.  Note that this fixed point set is a disjoint union of embedded circles. The group $\mathbb{Z}_2 \cong G_x/\langle \sigma \rangle$ acts on  $(M'')^\sigma$ 
and the set of fixed points is precisely $\Lambda_x$. The Euler characteristic of 
$(M'')^\sigma$ being zero implies that $\Lambda_x$ has an even number of points.
Since $|\Lambda_x|=|N_{G_x}(G_x)/G_x|$, then the group $N_{G_x}(G_x)/G_x$ has an
element of order 2, and therefore we may choose $g \in N_{G_x}(G_x)$ as a lift of this element of order 2. If $y:=gx$, then by construction we have $x=gy$.

If $k=2$, then there is only one isomorphism class of local representations. 
Let $\Gamma_{G_x}=\{p \in M'' | \mathrm{Stab}_{G_x}(p) \ne \{1\}\}$ be the ramification locus of this action.  Then $\Gamma_{G_x}$ is a properly embedded topological graph.  Because the action of $G_x$ on $\partial M''=\Sigma$ is free, $\Gamma_{G_x} \cap \Sigma$ is empty, and so in particular every vertex in this graph has valence $6$.  
The quotient graph $\Gamma_{G_x}/G_x$ resides in the quotient manifold $M''/G_x$, and its vertices are in bijection with the vertices of $\Gamma_{G_x}$, hence, in bijection with the points in $\Lambda_x$.  The vertices of $\Gamma_{G_x}/G_x$ all have valence 3.  Since twice the number of edges equals three times the number of vertices, we see that each connected component of $\Gamma_{G_x}/G_x$ has an even number of vertices, and hence the same is true for each connected component of $\Gamma_{G_x}$.  Note that this implies $|\Lambda_x|=|N_{G_x}(G_x)/G_x|$ is even.  Now as in the previous paragraph, we again choose $g \in N_{G_x}(G_x)$
lifting an element of order 2 in $N_{G_x}(G_x)/G_x$ and take $y :=gx$ with $x=gy$.

In both cases $k>2$ and $k=2$, note that $g^2 \in G_x$, and therefore the conjugation action of $g$ on $G_x$ squares to an inner automorphism of $G_x$.  This is especially helpful when $k=2$, \emph{i.e.}\ when $G_x \cong D_4 \cong \mathbb{Z}_2 \times \mathbb{Z}_2$, since in this case $\mathrm{Inn}(G_x) = 0$, and we can conclude that $g$ acts on $G_x$ by an automorphism of order either 1 or 2 (\emph{i.e.}\ never 3).

We now specify a preferred "rotation subgroup generator" of $G_x$.  When $k=2$ and $g$ conjugates $G_x$ by a non-trivial automorphism (necessarily of order 2, as just discussed above), then we take $\sigma$ in $G_x$ to be the unique non-trivial element of $G_x$ fixed by conjugation with $g$.  If $g$ conjugates $G_x$ trivially and there is not a loop in $\Gamma_{G_x}/G_x$ at $x$, then we take $\sigma$ to be an arbitrarily chosen element of $G_x$; if there is a loop at $x$, then we take $\sigma$ to correspond to the unique element in $G_x$ that does not stabilize points in the preimage of the interior of the loop (here the preimage can be taken with respect to $\Gamma_{G_x} \to \Gamma_{G_x}/G_x $).  When $k>2$, our preferred $\sigma$ is given essentially for free: pick either one of the two non-trivial elements of $G_x$ with minimal (unsigned) rotation angle (in its action on $N_x$) and call it $\sigma$.

With these choices for $\sigma$, in either the $k=2$ or $k>2$ case, we may parametrize $G_x$ as $G_x = \langle \sigma, \alpha \mid \sigma^k = \alpha^2 = 1, \alpha\sigma\alpha = \sigma^{-1} \rangle \cong D_{2k}$ for some arbitrarily chosen ``reflection" $\alpha$ in $G_x$. We also know that $g$ commutes with $\sigma$ whenever $k=2$ (because of how we picked $\sigma$), but
when $k>2$ it may be that $g\sigma g^{-1}=\sigma^{-1}$ since the local representations
around $x$ and $y$ are isomorphic. If this were the case, replace $g$ by $\alpha g$
and note that $\alpha g$ commutes with $\sigma$. Therefore we have found $g \in N_{G}(G_x)$ with
$gx=y \neq x$, $gy=x$ and $g\sigma g^{-1}=\sigma$.

  This in turn implies the following essential facts:
\begin{itemize}
\item when $k>2$, then no matter how $\sigma$ conjugates $G_x$, we must have that $g^2 = \sigma^l$ for some $0 \le l \le k-1$.
\item when $k=2$, then $\sigma$ must conjugate $G_x = \{1,\sigma,\alpha,\sigma\alpha\} \cong D_4$ by an automorphism that leaves $\sigma$ invariant.  Thus, $\sigma$ either commutes with all of $G_x$, or else swaps $\alpha$ and $\sigma \alpha$.
\begin{itemize}
\item If $g$ swaps $\alpha$ and $\sigma\alpha$, notice that $g$ does not commute with $\alpha$ or $\alpha\sigma$, hence $g^2$ (which $g$ \emph{does} commute with $g^2$) can not equal $\alpha$ or $\alpha\sigma$.  In other words, when $g$ acts non-trivially on $G_x$, then we must have $g^2 =1$ or $g^2=\sigma$. 
\item If $g$ commutes with all of $G_x$, then in principle we might have $g^2$ equal any element of $D_4$.  We will see below that in fact the only possibility is $g^2 = 1 \in D_4$.
\end{itemize}
\end{itemize}

Consider the group
\begin{align}
 K := \mathrm{Stab}_{G}\{x,y\} = \mathrm{Stab}_{N_G(G_x)}\{x,y\} = \langle \sigma, \alpha, g\rangle.
 \end{align}
Notice that by construction, if $h$ is any element of $G$ such that $hx = y$, then in fact $h \in K$.  (This is because $hx = gx$ implies $g^{-1}h \in G_x = G_y$, hence $g^{-1}h \in G_x$ and so $h \in gG_x \subset K$.)  In other words, we can build a $G$-equivariant matching on the orbit $G\cdot x$ by taking $\{x,y\}$ to be one pair in the matching, and inducing up to the entire orbit; the stabilizer of any edge in this matching is then conjugate to $G_x$.  Therefore, if we can build a $K$-equivariant handle that allows us to desingularize the action of $K$ on $N_x \sqcup N_y$, then we may induce this to a well-defined $G$-equivariant handle that desingularizes the action of $G$ on the entire orbit $G\cdot N_x$.

We will now classify the possibilities for how $K$ acts on $N_x \sqcup N_y$, and build non-singular handles for each possibility.  This will involve some casework, some of which depends on the integer $k\ge 2$ such that $G_x \cong D_{2k}$, and our success depends critically on the established fact that $g$ commutes with $\sigma$.

Notice that $K$ sits in an exact sequence
\begin{align} 1 \to G_x \to K \to K/G_x \to 1,\end{align}
where $K/G_x = \langle g \mod G_x \rangle = \mathbb{Z}_2$.  Recall that equivalence classes of such extensions can be placed in (non-canonical) bijection with the following pairs of data: homomorphisms $f: \mathbb{Z}_2 \to \mathrm{Out}(G_x)$ such that a certain canonically associated class in $H^3(\mathbb{Z}_2;\mathrm{Z}(G_x))$ vanishes, together with a class $\omega \in H_f^2(\mathbb{Z}_2; Z(G_x))$, where $\mathbb{Z}_2$ acts on the coefficients $Z(G_x)$ in a manner induced by $f$.

However, not all homomorphisms $f: \mathbb{Z}_2 \to \mathrm{Out}(G_x)$ will be pertinent to our situation, because (except in the case $k=2$ and $g$ commutes with $G_x$), we already know that $g^2=\sigma^l$ for some $0 \le l \le k-1$.  Let us use some group cohomology to constrain the possibilities for $K$ when $k>2$.

If $k$ is odd, then $Z(G_x) = \{0\}$, and so $H^2(K/G_x;Z(G_x)) = \{0\}$ and there is only one thing $K$ could possibly be given that $g^2 = \sigma^l$, namely:
\begin{align} K = \langle \sigma, \alpha, g \mid \sigma^k = \alpha^2 = 1, \alpha\sigma\alpha=\sigma^{-1}, g\sigma g^{-1} = \sigma, g\alpha g^{-1} = \alpha\sigma^l, g^2 = \sigma^l \rangle\end{align}
where $0 \le l \le k-1$.

If $k>2$ is even, then $Z(G_x) = \langle\sigma^{k/2}\rangle \cong \mathbb{Z}_2$ and the action of $K/G_x$ on the coefficients $\mathbb{Z}_2$ must be trivial no matter what $l$ is; therefore $H^2(K/G_x;Z(G_x))\cong\mathbb{Z}_2$ and we should expect two non-equivalent extensions for a given $l$.  These are precisely: 
\[\begin{aligned}
K_* &=\langle \sigma, \alpha, g \mid \sigma^k = \alpha^2 = 1, \alpha\sigma\alpha=\sigma^{-1}, g\sigma g^{-1} = \sigma, g\alpha g^{-1} = \alpha\sigma^l, g^2 = \sigma^l \rangle \\
K_\dagger &= \langle \sigma, \alpha, g \mid \sigma^k = \alpha^2 = 1, \alpha\sigma\alpha=\sigma^{-1}, g\sigma g^{-1} = \sigma, g\alpha g^{-1} = \alpha\sigma^{l+(k/2)}, g^2 = \sigma^l \rangle,
\end{aligned}
\]
where $0 \le l \le k-1$.

If $k=2$, rather than use group cohomology to upper bound the possibilities for $K$, we simply list the 6 known possibilities so far:
\[
\begin{aligned}
K_1 &=\langle \sigma, \alpha, g \mid \sigma^2 = \alpha^2 = 1, \alpha\sigma\alpha=\sigma, g\sigma g^{-1} = \sigma, g\alpha g^{-1} = \alpha\sigma, g^2 = 1 \rangle \qquad \mathrm{(!)}\\
& = \langle \alpha, g \alpha \mid   (g \alpha)^4=\alpha^2=1, \alpha(g\alpha)\alpha= (g\alpha)^{-1} \rangle \cong D_8 \\
K_2 &=\langle \sigma, \alpha, g \mid \sigma^2 = \alpha^2 = 1, \alpha\sigma\alpha=\sigma, g\sigma g^{-1} = \sigma, g\alpha g^{-1} = \alpha, g^2 = \sigma \rangle \qquad \mathrm{(!)} \\
&\cong \mathbb{Z}_2 \times \mathbb{Z}_4 \\
K_3 &=\langle \sigma, \alpha, g \mid \sigma^2 = \alpha^2 = 1, \alpha\sigma\alpha=\sigma, g\sigma g^{-1} = \sigma, g\alpha g^{-1} = \alpha, g^2 = \alpha \rangle \qquad \mathrm{(!)}  \\
&\cong K_2  \\
K_4 &=\langle \sigma, \alpha, g \mid \sigma^2 = \alpha^2 = 1, \alpha\sigma\alpha=\sigma, g\sigma g^{-1} = \sigma, g\alpha g^{-1} = \alpha, g^2 = \sigma\alpha \rangle \qquad \mathrm{(!)}  \\
&\cong K_2  \\
K_5 &=\langle \sigma, \alpha, g \mid \sigma^2 = \alpha^2 = 1, \alpha\sigma\alpha=\sigma, g\sigma g^{-1} = \sigma, g\alpha g^{-1} = \alpha\sigma, g^2 = \sigma \rangle\\
& = K_* \  (\mathrm{for} \ l=1) \cong D_8  \\
K_6 &=\langle \sigma, \alpha, g \mid \sigma^2 = \alpha^2 = 1, \alpha\sigma\alpha=\sigma, g\sigma g^{-1} = \sigma, g\alpha g^{-1} = \alpha, g^2 = 1 \rangle  \\
& = K_*   \  (\mathrm{for} \ l=0) \cong \mathbb{Z}_2 \times \mathbb{Z}_2 \times \mathbb{Z}_2    
\end{aligned}
\]
Each of these groups have order 8 as needed, so we see that no further constraints are needed to specify $K$ beyond specifying which of the two possible automorphisms fixing $\sigma$ we have $g$ act on $G_x$ by, and which element of $\sigma^2 \in G_x$ is.

All of the above listed possibilities for $K$ are based on naive algebra.  An algebraic classification of the possibilities for $K$ is not immediately equivalent to a geometric classification of the different possible \emph{faithful representations} of these $K$ with $K \to \mathrm{Iso}_+(B^3\sqcup B^3)$, which are, after all, what we need to desingularize.  In particular, we will see that the $K_1$ and $K_2$ (and therefore $K_3$ and $K_4$) in the case $k=2$ above marked with (!) have \emph{no} faithful representations into $\mathrm{Iso}_+(B^3\sqcup B^3)$.

To understand how these possible abstract structures of the extension $K$ relate to the geometry of the action of $K$ on $N_x \sqcup N_y \cong B^3 \sqcup B^3$, we parametrize so that $G_x$ acts on each copy of $B^3 \subset \mathbb{R}^3$ in the same standard way:
\begin{align}
\sigma(x,y,z) &= \left(\cos\left(\frac{2\pi l}{k}\right)x-\sin\left(\frac{2\pi l}{k}\right)y, \sin\left(\frac{2\pi l}{k}\right)x + \cos\left(\frac{2\pi l}{k}\right)y, z\right) \\
\alpha(x,y,z) &= \left(x,-y,-z\right)
\end{align}
with $0 < l <k $ and $l$ coprime with $k$.

Note that this standard action of $D_{2k}$ on $B^3$ is unique up to a sign, meaning any two faithful representations $D_{2k} \to \mathrm{SO}(3)$ that have $\sigma$ acting by rotation angle $\pm \frac{2\pi l}{k}$ are related by a conjugacy in $\mathrm{SO}(3)$.  With this, we see that the equivalence class of an isometric action of $K$ on $N_x \sqcup N_y$---\emph{when it exists}---is entirely determined (in the relevant sense, namely, up to conjugation by $\mathrm{Iso}_+(N_x \sqcup N_y)$) by the representation of $G_x$ on either component, and the diffeomorphism affected when $g$ swaps $N_x$ and $N_y$.

We will now show that none of the groups marked with a (!) in the $k=2$ case above is geometrically realizable.  It suffices to show this for $K_1$ and $K_2$.  In both cases, we assume without loss of generality that $\sigma$ and $\alpha$ act on $B^3 \sqcup B^3$ in the standard way shown above.

For $K_1$: the only possibilities for $g$ are $g(x,y,z)= (-y,x,z)$, $g(x,y,z)= (y,-x,z)$, $g(x,y,z)= (-y,-x,-z)$ or $g(x,y,z)= (y,x,-z)$, since
$g$ must leave the $z$-axis fixed and swap the $x$ and $y$ axis. The first two contradict $g^2 = 1$.  The second two contradict that $g$ commutes with $\alpha$.

For $K_2$: since $g$ commutes with all three generators, it leaves each axis fixed, and the possible actions are exactly: $g(x,y,z)= (x,-y,-z)$, $g(x,y,z)= (-x,y,-z)$ or  $g(x,y,z)= (-x,-y,z)$.  None of these squares to $\sigma$.

Finally, we will show that all remaining $K$ are geometrically realizable while simultaneously achieving our most important goal: a description of the desingularizing handle we are after for each possibility.

For each of them, we may define a faithful 4-dimensional real representation of $K = \langle \sigma, \alpha, g \rangle$ as follows:
\begin{align}
\sigma(x,y,z,t)=&(\cos(\tfrac{2\pi l}{k})x  -\sin(\tfrac{2\pi l}{k})y,\sin(\tfrac{2\pi l}{k})x + \cos(\tfrac{2\pi l}{k})y,z,t)\\
\alpha(x,y,z,t)=&(x,-y,-z,t)\\
g(x,y,z,t)=&(\cos(\tfrac{\pi j}{k})x  -\sin(\tfrac{ \pi j}{k})y,\sin(\tfrac{ \pi j}{k})x + \cos(\tfrac{\pi j}{k})y,-z,-t).
\end{align}
Here $j=l$ if $k$ is odd, and $j=l$ or $l+(k/2)$ if $k>2$ is even and $K=K_*$ or $K=K_\dagger$ respectively. For the groups $K_5$ and $K_6$ we take 
$j=1$ and $j=2$ respectively. Note that we are continuing to assume (without loss of generality) that $\sigma$ is the element in $G_x$ that acts on $N_x$ as rotation by the angle $\frac{2 \pi l}{k}$.

Clearly $(x,y,z,t)$ is a fixed point of $g$ only if $z=t=0$.  In the two cases with $k=2$, observe that for $K_5$ we have
\begin{align} g(x,y,z,t)=(-y,x,-z,-t) \end{align}
and for the group
$K_6$ we have
\begin{align} g(x,y,z,t)=(-x,-y,-z,-t),\end{align}
so other than the zero point $(0,0,0,0)$, $g$ has no fixed points at all in $\mathbb{R}^4$.  Thus $(x,y,0,0) \ne (0,0,0,0)$ is a fixed point of $g$ only if $k>2$ and $j=l=0$.  Non-trivial powers of $\sigma$ never share a non-zero fixed point with $g$, as $\sigma$ acts freely on the plane $z=t=0$.  We conclude that any non-zero point in the $z=t=0$ plane has either a trivial stabilizer, a cyclic stabilizer generated by an element of $G_x$, or a stabilizer of the form $\langle \sigma^p \alpha, g\rangle \cong D_4$ for some $0 \le p \le k-1$, and moreover, this third case can only occur when $k>2$ and $j=l=0$.  This latter fact---that non-cycle stabilizers occur in this plane only when $k>2$---is essential to the remainder of our argument.

Now consider the action of $K$ on the unit sphere $S^3 \subset \mathbb{R}^4$, \emph{i.e.}\ on 
\begin{align}
 S^3 = \{(x,y,z,t) \in \mathbb{R}^4 \colon x^2 +y^2+z^2+t^2=1 \}.
 \end{align}
The isotropy group of each point $(0,0,0,-1)$ and $(0,0,0,1)$ is the dihedral group $G_x = \langle \sigma,\alpha\rangle \cong D_{2k}$ and $g$ swaps these two points.   Moreover, every other point in $S^3 \setminus \{(0,0,0,-1), (0,0,0,1)\}$ has either trivial or cyclic isotropy, with some minor exceptions: when $k>2$ and $j=l=0$, there are points in $S^3 \setminus \{(0,0,0,-1), (0,0,0,1)\}$ with isotropy isomorphic to $D_4$. Denote by $W$ a small $K$-equivariant ball around the union of $(0,0,0,-1)$ and $(0,0,0,1)$ and remove it from $S^3$. Attach the $G$-equivariant handle $G \times_K (S^3 \backslash W)$ to the boundary of the $G$ equivariant normal neighborhood of $G \cdot x$ on $M$.  For $k>2$,  the resulting 3-manifold has fewer points with dihedral isotropy $D_{2k}$, although it may create new points with $D_4$ isotropy.   Attach the handles inductively for all dihedral isotropies isomorphic to $D_{2k}$ with $k>2$ (picking maximal such isotropies at every step) and all isomorphism classes of faithful irreducible representatios, and arrive at a manifold whose isotropies are only cyclic or dihedral of order $4$.  Now repeat step 2 of the proof to arrive at a manifold with only cyclic isotropies and dihedral isotropies of order 4 with, moreover, the property that for any $x$ and $y$ with $G_x = G_y \cong D_4$, we know $x$ and $y$ are in the same $G$ orbit. Finally, return to step 3 and desingularize any remaining $D_4$ isotropies as we did in the case with $k>2$.  Since the handles we have constructed for the $D_4$ singularities have no non-cyclic isotropies on their interior, attaching them to the remaining $D_4$ singularities gives our final manifold $M'''$ with only cyclic isotropies.
\end{proof}

We are now ready to show which surfaces with free $G$-actions bound equivariantly.

\begin{theorem} \label{theorem free actions that bound}
Let $G$ be a finite group. Then the oriented and unitary
equivariant bordism of surfaces with free $G$-actions fits into the exact sequence
\begin{equation} \label{exact sequence for B0(G)}
{\Omega}_3^{G}\{\cala G,\{1\}\} \stackrel{\partial}{\to} \widetilde{\Omega}_2(BG) \to \tilde{B}_0(G) \to 0.
\end{equation}
\end{theorem}
\begin{proof}
Let us first show that the image of the boundary map consists of toral classes in $H_2(BG,\IZ)$, that is,
homology classes coming from the image of maps of tori $S^1 \times S^1 \to BG$.

Let $M$ be a 3-dimensional $G$-manifold (oriented or unitary) whose boundary $\partial M$ has a free $G$-action;
in the oriented case take $M$ as shown in Lemma \ref{Lemma blow up}
Note that if the ramification locus $\overline{M}$ is not empty, then it is a smooth oriented 1-dimensional manifold; in the unitary case
this follows from the fact that fixed points of all non-trivial subgroups can only have complex codimension 0 or 1.

 If $\overline{M}$ is empty
then $M$ has a free $G$ action  and therefore the boundary surface $(\partial M)/G$ bounds. If $\overline{M}$ is not empty
we may consider the $G$-equivariant tubular neighborhood $N$ of $\overline{M}$ in $M$. The manifolds $N$ and $M$ define
the same bordism class since on $M -N$ the action of $G$ is free, and therefore $\partial M$ and $\partial N$ are cobordant. The tubular neighborhood $N$ is homeomorphic
to the unit ball bundle $B \nu$ of the normal bundle $\nu$ of $\overline{M}$ in $M$. The sphere bundle $S\nu$ defines the $S^1$-principal bundle
$S^1 \to S\nu \to \overline{M}$ and since every circle bundle over the circle is topologically a torus, then the sphere bundle $S\nu$ is homeomorphic to a disjoint union of 2-dimensional tori. Hence $\partial N$ is a disjoint union of 2-dimensional tori, and its quotient $\partial N / G $ is a 
torus (since $M/G$ is connected and $\chi(\partial N/G) = \chi(\partial N) /|G|$). 
Hence we have now proved that the image of the boundary map $\partial$ of \eqref{exact sequence for B0(G)} consists only of toral classes in $H_2(BG, \IZ)$.

Now let us show the converse, namely, that any toral class in $H_2(BG, \IZ)$ lies in the image of the boundary map of \eqref{exact sequence for B0(G)}. Take any toral class defined by a homomorphism $\varphi : \IZ \times \IZ \to G$ and denote by $A:=\varphi(\IZ\times\{0\})$ and
$C:=\varphi(\{0\} \times \IZ)$ the cyclic subgroups of $G$ that define the toral class. Denote by $a := \varphi(1,0)$ and $c:=\varphi(0,1)$
the generators of $A$ and $C$ respectively. 

Let $N_GA$ be the normalizer of $A$ in $G$ and note that $C$ is a subgroup of the normalizer. 
Denote by $\iota$ and $\bar{\iota}$ the homomorphism $\iota: \IZ \to N_GA$, $\iota(n):=c^n$ and the homomorphism to the quotient $\bar{\iota}: \IZ \to W_A$.
Consider the irreducible representation $\rho : A \to U(1)$, $\rho(a):=e^\frac{2\pi i}{|A|}$, and define the $U(1)$ extension
of $W_A$ by the exact sequence of groups
\begin{equation}
U(1) \to U(1) \times_A N_GA \to W_A
\end{equation}
where $U(1) \times_A N_GA$ is defined by the equivalence relation $(\lambda \rho(\alpha), g) \sim (\lambda, \alpha g)$
for all $\alpha \in A$, $\lambda \in U(1)$ and $g \in N_GA$.

Consider the homomorphism $\tilde{\iota}:\IZ \to U(1) \times_A N_GA$, $\tilde{\iota}(n) := [(1,\iota(c^n))]$, and note that its classifying
map 
\begin{equation}
B\tilde{\iota} : S^1 \to B(U(1) \times_A N_GA)
\end{equation}
factors through the classifying map $B\bar{\iota}: S^1 \to BW_A$.

Let $E := (B\tilde{\iota})^* E(U(1) \times_A N_GA)$ be the pullback of the universal bundle and note two things. First, $E$
is a principal $U(1) \times_A N_GA$-bundle over the circle $S^1$, and therefore it is a surface. Second, the canonical homomorphism 
\begin{equation}
N_GA \to U(1) \times_A N_GA, \ \ g \mapsto [(1,g)],
\end{equation}
induces a free action of $N_GA$ on $E$. Now it is straightforward to notice that the homology class of the surface $E/N_GA \to BN_GA \to BG$ in
$\widetilde{\Omega}_2(BG)$
agrees with the homology class defined by $B\varphi_*[S^1\times S^1]$.

We still need to show that the surface $E$ equivariantly bounds. Take the quotient $F = E/U(1)$
and note that $F$ is homeomorphic to $(B\bar{\iota})^* EW_A$; hence $F$ is the principal 
$W_A$-bundle over the circle that $\bar{\iota}$ defines (see the following commutative diagram).
\begin{align}
\xymatrix @R=0.15pc @C=0.5pc{
U(1) \ar[rd] \ar@{-}[dd] &&&&\\
&U(1)\times_AN_GA \ar@{-}[rr] \ar[dd]  \ar[rd] & & U(1)\times_AN_GA \ar@{-}[d]  \ar[rd]& \\
U(1) \ar[dr] &  & W_A \ar@{-}[rr]  \ar[dd] & \ar[d] & W_A \ar[dd]\\
&E \ar@{-}[r] \ar[rd] \ar[dd]&  \ar[r]& E(U(1)\times_A N_GA) \ar@{-}[d]  \ar[rd]& \\
 & & F \ar[rr] \ar[dd] &{} \ar@{->}[d] & EW_A \ar[dd] \\
&S^1 \ar@{-}[r]^{B\tilde{\iota}} \ar@2{-}[rd] & \ar@{->}[r]&  B(U(1)\times_A N_GA)   \ar[rd]& \\
 & & S^1 \ar[rr]^(.15){B\bar{\iota}} & & BW_A 
}
\end{align}

We have that $E$ is a principal $U(1)$-bundle over $F$ and
therefore we may take the associated complex vector bundle
\begin{equation}
\IC \to \IC \times_{U(1)} E \to F.
\end{equation}
The unit bundle $D(\IC \times_{U(1)} E)$ is a unitary manifold endowed with the action of $N_GA$, whose
boundary, the sphere bundle $S(\IC \times_{U(1)} E)$, is homeomorphic to $E$:
\begin{equation}
\partial  \left( D(\IC \times_{U(1)} E) \right)  = S(\IC \times_{U(1)} E) \cong E.
\end{equation}

Therefore we have just proved that 
\begin{equation}
[G\times_{N_GA} D(\IC \times_{U(1)} E) ]\stackrel{\partial}{\to} [E/N_GA] =  B\varphi_*[S^1\times S^1],
\end{equation}
thus showing that any toral class in $\widetilde{\Omega}_2(BG)$ equivariantly bounds.
\end{proof}

Now we can put all the pieces together to understand the torsion of the equivariant bordism group of surfaces.

\subsection{Torsion of the equivariant bordism group of surfaces}
\label{ss:torsion}

By Corollary \ref{c:nofixed} we have
\begin{equation}
\overline{\Omega}_2^{G} = \Ker(\phi_2) = \Tor_\IZ (\Omega^{G}_2).
\end{equation}
Let us now determine explicitly these torsion subgroups.

\begin{theorem} \label{theorem torsion subgroup bordism}
Let $G$ be a finite group. Then there is a canonical isomorphism
\begin{equation}
  \bigoplus_{(K)}  \tilde{B}_0(W_K) \cong  \Tor_\IZ (\Omega^{G}_2)
\end{equation}
where $(K)$ runs over all conjugacy classes of subgroups of $G$, $W_K=N_GK/K$ and $\tilde{B}_0(W_K)$ is 
the homology version of the Bogomolov multiplier of the group $W_K$.
\end{theorem}

\begin{proof}

Denote by $\Gr_*  \Tor_\IZ (\Omega^{G}_2)$ the
associated graded groups of the $G$-equivariant, unitary or oriented,
bordism groups of surfaces that are induced by the Conner-Floyd spectral sequence of the families of subgroups of \eqref{dense filtration}. 
 Lemma 
\ref{l:gr} and Theorem \ref{theorem free actions that bound} imply that
\begin{align}
\Gr_p  \Tor_\IZ (\Omega^{G}_2) \cong \tilde{B}_0(W_{K_p}),
\end{align}
and since all consecutive pairs of families are adjacent, we obtain the graded isomorphism
\begin{align} \label{graded isomorphism Tor omega}
\Gr_*  \Tor_\IZ (\Omega^{G}_2) \cong \bigoplus_{(K)}  \tilde{B}_0(W_K).
\end{align}

Now, for a fixed conjugacy class of subgroups $(K)$, the canonical map 
\begin{align} 
\widetilde{\Omega}_2(BW_{K}) \to \Omega_2^G, \ \ \ \ \Sigma/W_K  \mapsto G \times_{N_GK} \Sigma,
\end{align}
which sends the quotient space of a surface $\Sigma$ by the free $W_{K}$-action to the surface
with $G$-action whose isotropy groups lie in $(K)$, factors through $\tilde{B}_0(W_{K})$, thus producing a canonical homomorphism
\begin{align}
\tilde{B}_0(W_{K}) &\to \Omega_2^G.
\end{align}
Bundling up all these homomorphisms we obtain a canonical map
\begin{align}
\bigoplus_{(K)}  \tilde{B}_0(W_K) \to \Tor_\IZ (\Omega^{G}_2)
\end{align}
which becomes an isomorphism since it is compatible with the graded isomorphism of \eqref{graded isomorphism Tor omega}.
\end{proof}

In particular, if $G$ is a group whose Bogomolov multipliers vanish for all groups $W_K$ with $K$
a non-trivial subgroup, then $\Tor_\IZ (\Omega^{G}_2) \cong \tilde{B}_0(G)$.
This is the case whenever $G$ is one of the smallest  $p$-groups with non-trivial Bogomolov multiplier.
In the last section we present two $p$-groups of this kind.

We are now ready to provide an explicit calculation of the unitary and oriented 
equivariant bordism group of surfaces. Assembling Theorem \ref{th:Rowlett}, Proposition \ref{p:23AGPG} and
Theorem \ref{theorem torsion subgroup bordism} we obtain the following result.
\begin{theorem} \label{theorem decomposition equivariant bordism of surfaces}
Let $G$ be a finite group. Then the unitary and oriented
equivariant bordism of surfaces canonically decompose as follows:
\begin{align}
\Omega_2^{U,G} \cong& \bigoplus_{(K)}\left(\tilde{B}_0(W_K) \oplus \Omega_2^U  \oplus  \left( \bigoplus_{ \Irr_{\mathbb{C}}^1(K)} \mathbb{Z} \right)^{W_K}  \right) ,\\
\Omega_2^{SO,G} \cong& \bigoplus_{(K)} \left( \tilde{B}_0(W_K)   \oplus  \left( \bigoplus_{ \Irr_{\mathbb{C}}^1(K)_\mathbb{C}/\mathrm{conj}} \mathbb{Z} \right)^{W_K}  \right) .
\end{align}
Here $(K)$ runs over the conjugacy classes of subgroups of $G$, $W_K$ is the Weyl group $N_GK/K$, 
$ \Irr_{\mathbb{C}}^1(K)$ is
 the set of 1-dimensional non-trivial irreducible complex 
representations of $K$ endowed with the natural $W_K$ action, and $\Irr_{\mathbb{C}}^1(K)_\mathbb{C}/\mathrm{conj}$  denotes the representations of complex type modulo complex conjugation.
\end{theorem}

\section{2-dimensional SK-groups of classifying spaces}
J\"anich in \cite{Jaenich1,Jaenich2} started the study of the characterization of  invariants with the additivity property of the Euler characteristic and the signature under cutting and pasting of manifolds.

Karras and Kreck in their Diplom thesis extended the ideas of J\"anich to cutting and pasting in the bundle situation. The book  \cite{KKNO} presented and simplified these results with the definition of the $SK$-groups  of a space (cutting and pasting groups from the German \textit{Schneiden und Kleben}).  Later  Neumann \cite{Neumann} completely calculated the 2-dimensional $SK$-groups of a space in terms of what is now known as the Bogomolov multiplier of its fundamental group.
We recall in this section the main results  of \cite{KKNO} and \cite{Neumann} that allow us to relate the $SK$-relation with the equivariant bordism relation on surfaces with free actions.

The \textit{Schneiden und Kleben}  groups $SK_*(X)$ of a space are defined as the Grothendieck group of the semigroups obtained by defining the $SK$-equivalence on the class of continuous maps from oriented $n$-dimensional manifolds  to $X$ \cite{KKNO}.

The $SK$-relation is defined as follows:  given $(M_i,f_i)$ with $f_i :  M_i \rightarrow X$, we say that $(M_1,f_1)$ and $(M_2,f_2)$ are related by cutting and pasting along $\partial N$ if $M_1 = N \cup_\phi -N'$,$
M_2 = N \cup_\psi -N'$ and there are homotopies $f_1 \mid_N \simeq f_2\mid_N$, $f_1 \mid_{N'} \simeq f_2\mid_{N'}$.

The \textit{Schneiden und Kleben} bordism groups $\overline{SK}_n(X)$ of a space are defined as the quotient of the oriented bordism groups by the equivalence relation generated by the $SK$-relation:
\begin{equation}
\overline{SK}_*(X) = \Omega^{SO}_*(X) /  \sim.
\end{equation}

The group $\overline{SK}_2(BG)$ can be interpreted as the bordism group of surfaces with free $G$-actions modulo the $SK$-relations.

The following results summarize the main properties of the $SK$-relation \cite[Lem. 1.5 \& 1.6]{KKNO}.

\begin{enumerate}
\item  Any $f: S^1 \rightarrow X$ is zero in $SK_1(X)$ .
\item  If $M$ fibers over $S^n$ with fiber $F$ then for  any $f: M \rightarrow X$,   in $SK_*(X)$ we have 
\begin{equation}
[M,f] = [S^n,*] [F,f|_F].
\end{equation}
\item If $[M_2,f_2]$ is obtained from $[M_1,f_1]$ by surgery of type $(k+1,n-k)$, then in $SK_*(X)$
\begin{equation}
[M_1,f_1] + [S^n,*] = [M_2,f_2] + [S^k \times S^{n-1},*].
\end{equation}
\end{enumerate}

Now, if 
$I_*$ denotes the subgroup of $SK_*(X)$ generated by the spheres with constant maps to $X$, which is isomorphic to the integers, we have:
\begin{theorem}\cite[Thm. 1.1]{KKNO}
For a connected space $X$, there is the following exact sequence 
\begin{equation}
0 \rightarrow I_* \rightarrow SK_*(X) \rightarrow \overline{SK}_*(X) \rightarrow 0,
\end{equation}
which is moreover split. The map $\frac{\chi-\tau}{2} : SK_n(X) \rightarrow \mathbb{Z} \cong I_n$ gives the splitting.
\end{theorem}

The groups $\overline{SK}_*(X)$ 
fit into short exact sequences whose middle terms are the oriented bordism
groups.

\begin{theorem}\cite[Thm 1.2]{KKNO} \label{theorem exact sequence of SK}
Let $F_n(X)$ be the submodule of  $\Omega^{SO}_n(X)$ generated by all elements which have a representative that fibers over $S^1$. Then $F_n(X)$ fits into the short exact sequence
\begin{equation}
0 \rightarrow F_*(X) \rightarrow \Omega^{SO}_*(X) \rightarrow \overline{SK}_*(X) \rightarrow  0.
\end{equation}
\end{theorem}
This  theorem  follows from the observations that any manifold that fibers over $S^1$ gives a class that is zero in $\overline{SK}_*(X)$,
 and that the kernel of the homomorphism  $\Omega^{SO}_*(X) \rightarrow \overline{SK}_*(X)$ consists of mapping tori. 
The key lemma for the opposite inclusion asserts that if $(M_1,f_1)$ is obtained from $(M_2,f_2)$ by cutting and pasting along $N$, 
 then in $\Omega^{SO}_*(X)$ the class of $(N \cup_\phi -N',f_1) - (N \cup_\psi -N',f_2)$ is equal to the mapping torus of the  diffeomorphism of $\partial N$, $\phi^{-1} \circ \psi$.
 Any mapping torus fibers over $S^1$ and any fibration over $S^1$ is a mapping torus.

In dimension $0$ and $1$ the groups $\overline{SK}_n(X)$ are trivial.  In dimension $2$ the oriented manifolds that fiber over the circle are tori. Therefore by Theorem \ref{theorem exact sequence of SK} we obtain the following result:
\begin{theorem}\cite[Thm. 2]{Neumann} Let $G$ be a discrete group. Then the 2-dimensional $\overline{SK}$-group of $BG$ is
isomorphic to the Bogomolov multiplier of $G$, i.e.,
\begin{equation}
\overline{SK}_2(BG) \cong \tilde{B}_0(G).
\end{equation}
\end{theorem}
Reinterpreting the $SK$-groups of $BG$ in view of our previous results, we know by Theorem
\ref{theorem free actions that bound} that an element of $\overline{SK}_2(BG)$ is zero whenever the
associated $G$-cover of the surface is the boundary of a three dimensional manifold with a $G$-action.
By Theorem \ref{theorem exact sequence of SK}  we have that $SK_2(BG)\cong \mathbb{Z} \oplus \tilde{B}_0(G)$, and therefore a surface $\Sigma \to BG$   is zero in the group $SK_2(BG)$ whenever
  the Euler characteristic of $\Sigma$ is 0 and the $G$-cover $\widetilde{\Sigma}$ of $\Sigma$
 is the boundary of a three dimensional manifold with a $G$-action.
  
 It would be interesting to explore the relation of this work with the higher dimensional $SK$-groups of classifying spaces.

\section{Small groups with non-trivial Bogomolov multiplier}
\label{s:small}

We conclude this work by presenting some explicit examples of groups with non-trivial Bogomolov multiplier which induce non-trivial torsion subgroups in the equivariant bordism groups of surfaces.  Some of the calculations were done with the help of the {\it Homological Algebra Programming} package for GAP \cite{GAP4}.

\subsection{2-group of size 64}
The smallest groups with non-trivial Bogomolov multiplier are 2-groups of order 64.  There are 9 of them and all are in the same isoclinism class.  By \cite[Thm 1.2]{Moravec-isoclinism} they all have isomorphic Bogomolov multipliers and in this case it is the group $\IZ/2$.  Among the 9 isoclinic groups we chose to study the group 
\begin{equation}
C_8 \rtimes Q_8 
\end{equation}
which is the semidirect product of the group of quaternions $Q_8$ with the cyclic group $C_8$ of order $8$;
this group is denoted
\begin{center}
{\fontfamily{qcr}\selectfont SmallGroup(64,182)}
\end{center}
in the GAP small groups library.  Consider the presentations of the groups
 $Q_8 = \langle a , b \colon a^2=b^2, aba^{-1}=b^{-1} \rangle$ and $C_8 = \langle c \colon c^8=1 \rangle$ and the action
of $Q_8$ on $C_8$ given by the equations
\begin{equation}
a \cdot c = c^3,  \ \ \  b \cdot c = c^5,  \ \ \  (ab) \cdot c = c^7. 
\end{equation}
 Since $H^2(C_8, \IC^*)=0= H^2(Q_8, \IC^*)$, we know by the Lyndon-Hochschild spectral sequence that
\begin{equation}
H^2(C_8 \rtimes Q_8, \IC^*) \cong H^1(Q_8, H^1(C_8, \IC^*)) .
\end{equation}
Denote $\widehat{C}_8 := \Hom(C_8, \IC^*)= H^1(C_8, \IC^*)$ and let $\widehat{C}_8 = \langle \rho \colon \rho^8=1 \rangle$
with $\rho(c) = e^{\frac{2\pi i}{8}}$. Take the first two terms of the complex $C^*(Q_8,\widehat{C}_8)$ 
\begin{equation}
\widehat{C}_8 \stackrel{\delta}{\to} \Map(Q_8,\widehat{C}_8)
\end{equation}
and note that 
\begin{align}
\delta(\rho^k)(a^\pm) = \rho^{-2k}, \ &\ \delta(\rho^k)(b^\pm) = \rho^{4k},\\
 \delta(\rho^k)((ab)^\pm)= \rho^{2k}, \ &  \ \delta(\rho^k)(a^2)= \rho^{0}. \nonumber
\end{align}
On the other hand, take the 1-cocycle $F: Q_8 \to \widehat{C}_8 $ defined by the equations
\begin{align}
F(a^\pm) = \rho^2, \  & \ F(b^\pm) = \rho^0,\\ F((ab)^\pm) = \rho^2, \ & \ F(a^2) = \rho^0,
\end{align}
and note that $F$ does not bound but $F^2= \delta(\rho^2)$. We have therefore that
\begin{equation}
H^1(Q_8, \widehat{C}_8) \cong \langle [F] \colon [F^2]=0 \rangle \cong \IZ/2.
\end{equation}
Now, any abelian subgroup of $C_8 \rtimes Q_8 $ splits as a semi-direct product of abelian groups $C \rtimes A$ with $C \subset C_8$ and $A \subset Q_8$. Since $A$ can only be $\IZ/4$ or $\IZ/2$, it is now straightforward to check that $[F]|_{C \rtimes A}=0$.
Hence $[F]$ is the generator of the Bogomolov multiplier of $C_8 \rtimes Q_8 $ and we have that

\begin{equation}
\overline{\Omega}_2^{U, C_8 \rtimes Q_8 } \cong \overline{\Omega}_2^{SO, C_8 \rtimes Q_8 } \cong  \IZ/2.
\end{equation}

Finally, with the explicit description of $F$ we can define a surface $\Sigma_2$ of genus $2$ which defines the generator of $\widetilde{\Omega}_2^{U}( B(C_8 \rtimes Q_8))$. Consider the presentation of the fundamental group of the surface
\begin{equation}  \label{fundamental group surface genus 2}
\pi_1(\Sigma_2)= \langle x,y,z,w \colon [x,y][z,w]=1 \rangle
\end{equation}
and define the following assignment
\begin{align}
\Phi: \pi_1(\Sigma_2) & \to C_8 \rtimes Q_8\\
x \mapsto a, \ \ y \mapsto c, \ & \  z \mapsto ab, \ \ w \mapsto c, \nonumber
\end{align}
which induces a surjective homomorphism since 
\begin{equation}
\Phi([x,y][z,w])=aca^{-1}c^{-1}(ab)c(ab)^{-1}c^{-1}= c^3c^{-1}c^7c^{-1}=c^{0}.
\end{equation}

The homomorphism $\Phi$ induces a map $B\Phi : \Sigma_2 \to B(C_8 \rtimes Q_8)$, and from the construction
above of $F$, we deduce that $B\Phi_*[\Sigma_2]$ generates  the group $H_2( B(C_8 \rtimes Q_8), \IZ)$. 

Hence
the surface 
\begin{equation}
\widetilde{\Sigma}:= (B\Phi)^* E(C_8 \rtimes Q_8)
\end{equation}
 is a unitary surface with a free action of $C_8 \rtimes Q_8$
which does not equivariantly bound. 

By Theorem \ref{theorem torsion subgroup bordism} the class of
$\widetilde{\Sigma}$ is the generator of the torsion subgroup of ${\Omega}_2^{SO, C_8 \rtimes Q_8 }$:
\begin{equation}
\Tor_\IZ {\Omega}_2^{SO, C_8 \rtimes Q_8 } = \langle [\widetilde{\Sigma}]  \rangle \cong \IZ/2.
\end{equation}

To make sure that the first Chern number vanishes,
we take the bordism class
\begin{equation}
[\widetilde{\Sigma}] - [(C_8 \rtimes Q_8) \times \Sigma_2] \in \overline{\Omega}_2^{U, C_8 \rtimes Q_8 } \cong \IZ/2
\end{equation}
and by Theorem \ref{theorem torsion subgroup bordism} we conclude that this class is indeed the generator of the torsion subgroup of 
${\Omega}_2^{U, C_8 \rtimes Q_8 }$:
\begin{equation}
\Tor_\IZ {\Omega}_2^{U, C_8 \rtimes Q_8 } = \langle [\widetilde{\Sigma}] - [(C_8 \rtimes Q_8) \times \Sigma_2] \rangle \cong \IZ/2.
\end{equation}

\subsection{3-group of size 243}

The smallest 3-groups with non-trivial Bogomolov multiplier are of order 243, and the three of them are isoclinic
with Bogomolov multiplier the group $\IZ/3$.
We chose to study the group
\begin{equation}
G:=(C_9 \rtimes C_9) \rtimes C_3
\end{equation}
which is defined by the presentation
\begin{equation}
G=\langle a , b,c \colon a^3=c^3, a^9=b^9=1, [a,b]=c^8b^6, [b,c]=a^3, [a,c]=b^3c^6 \rangle.
\end{equation}
The left $C_9$ is generated by $c$, the middle $C_9$ by $b$ and the right $C_3$ by $ab$, and their corresponding actions are
\begin{equation}
bcb^{-1}=c^4,  \ \ (ab)b(ab)^{-1}=c^8b^7, \ \ (ab)c(ab)^{-1}=cb^3.
\end{equation}
This group corresponds to the small group 
\begin{center}{\fontfamily{qcr}\selectfont
SmallGroup(243,30)}\end{center} in the small groups library of GAP \cite{GAP4}.

The second page of the Lyndon-Hochschild spectral sequence has for terms
\begin{equation}
H^2(C_9 \rtimes C_9, \IC^*)^{C_3} = 0, \ \ H^1(C_3, H^1(C_9 \rtimes C_9, \IC^*))= \IZ/3, \ \ H^2(C_3,\IC^*)=0
\end{equation}
where the middle term encodes the information of the Bogomolov multiplier.

Consider the surface $\Sigma_2$ of genus 2 as in \eqref{fundamental group surface genus 2} and define the 
assignment
\begin{align}
\Phi: \pi_1(\Sigma_2) & \to (C_9 \rtimes C_9) \rtimes C_3\\
x \mapsto a, \ & \ y \mapsto b^6, \  \  z \mapsto c, \ \ w \mapsto b, \nonumber
\end{align}
which induces a surjective homomorphism since $[a,b^6]=a^3$, $[c,b]=a^6$ and
\begin{equation}
\Phi([x,y][z,w])=[a,b^6][c,b]= 1.
\end{equation}
The map $B\Phi: \Sigma_2 \to B((C_9 \rtimes C_9) \rtimes C_3)$ generates the Bogomolov multiplier and therefore
the surface
$\widetilde{\Sigma} := (B\Phi)^*E((C_9 \rtimes C_9) \rtimes C_3)$ generates the torsion subgroup of 
the equivariant oriented bordism group of surfaces
\begin{equation}
\Tor_\IZ {\Omega}_2^{SO, (C_9 \rtimes C_9) \rtimes C_3} = \langle [\widetilde{\Sigma}]  \rangle \cong \IZ/3.
\end{equation}
In the unitary case we have
\begin{equation}
\Tor_\IZ {\Omega}_2^{U, (C_9 \rtimes C_9) \rtimes C_3} = \langle [\widetilde{\Sigma}] - [(C_9 \rtimes C_9) \rtimes C_3 \times \Sigma_2] \rangle \cong \IZ/3.
\end{equation}

The surface $\widetilde{\Sigma}$ is a surface of genus 486 with a free action of $(C_9 \rtimes C_9) \rtimes C_3$
which does not equivariantly bound.

\bibliographystyle{alpha} 
\bibliography{Bordism-bibliography}

\end{document}